\documentclass[a4paper,leqno,12pt]{amsart}
\usepackage{amsmath,amsthm}
\usepackage{amssymb}
\usepackage[all]{xy}
\SelectTips{cm}{12}
\UseTips{}
\usepackage{xspace}
\usepackage{bm}
\usepackage{amstext}
\usepackage{amsfonts}
\usepackage[mathscr]{euscript}
\usepackage{amscd}
\usepackage{latexsym}
\usepackage{amssymb}
\usepackage{enumerate}
\begin{document}
%
%
\theoremstyle{plain}
\swapnumbers
    \newtheorem{thm}[figure]{Theorem}
    \newtheorem{prop}[figure]{Proposition}
    \newtheorem{lemma}[figure]{Lemma}
    \newtheorem{keylemma}[figure]{Key Lemma}
    \newtheorem{corollary}[figure]{Corollary}
    \newtheorem{fact}[figure]{Fact}
    \newtheorem{subsec}[figure]{}
    \newtheorem*{propa}{Proposition A}
    \newtheorem*{thma}{Theorem A}
    \newtheorem*{thmb}{Theorem B}
    \newtheorem*{thmc}{Theorem C}
    \newtheorem*{thmd}{Theorem D}
\theoremstyle{definition}
    \newtheorem{defn}[figure]{Definition}
    \newtheorem{example}[figure]{Example}
    \newtheorem{examples}[figure]{Examples}
    \newtheorem{notation}[figure]{Notation}
    \newtheorem{summary}[figure]{Summary}
\theoremstyle{remark}
        \newtheorem{remark}[figure]{Remark}
        \newtheorem{remarks}[figure]{Remarks}
        \newtheorem{warning}[figure]{Warning}
    \newtheorem{assume}[figure]{Assumption}
    \newtheorem{ack}[figure]{Acknowledgements}
\renewcommand{\thefigure}{\arabic{section}.\arabic{figure}}
%
%
%
\newenvironment{myeq}[1][]
{\stepcounter{figure}\begin{equation}\tag{\thefigure}{#1}}
{\end{equation}}
\newcommand{\myeqn}[2][]
{\stepcounter{figure}\begin{equation}
     \tag{\thefigure}{#1}\vcenter{#2}\end{equation}}
\newcommand{\mydiag}[2][]{\myeq[#1]\xymatrix{#2}}
\newcommand{\mydiagram}[2][]
{\stepcounter{figure}\begin{equation}
     \tag{\thefigure}{#1}\vcenter{\xymatrix{#2}}\end{equation}}
\newcommand{\mysdiag}[2][]
{\stepcounter{figure}\begin{equation}
     \tag{\thefigure}{#1}\vcenter{\xymatrix@R=25pt@C=-25pt{#2}}\end{equation}}
\newcommand{\myrdiag}[2][]
{\stepcounter{figure}\begin{equation}
     \tag{\thefigure}{#1}\vcenter{\xymatrix@R=25pt@C=23pt{#2}}\end{equation}}
\newcommand{\mytdiag}[2][]
{\stepcounter{figure}\begin{equation}
     \tag{\thefigure}{#1}\vcenter{\xymatrix@R=25pt@C=15pt{#2}}\end{equation}}
\newcommand{\myudiag}[2][]
{\stepcounter{figure}\begin{equation}
     \tag{\thefigure}{#1}\vcenter{\xymatrix@R=29pt@C=8pt{#2}}\end{equation}}
\newcommand{\myvdiag}[2][]
{\stepcounter{figure}\begin{equation}
     \tag{\thefigure}{#1}\vcenter{\xymatrix@R=19pt@C=3pt{#2}}\end{equation}}
\newcommand{\mywdiag}[2][]
{\stepcounter{figure}\begin{equation}
     \tag{\thefigure}{#1}\vcenter{\xymatrix@R=12pt@C=30pt{#2}}\end{equation}}
\newcommand{\myydiag}[2][]
{\stepcounter{figure}\begin{equation}
     \tag{\thefigure}{#1}\vcenter{\xymatrix@R=35pt@C=30pt{#2}}\end{equation}}
\newcommand{\myzdiag}[2][]
{\stepcounter{figure}\begin{equation}
     \tag{\thethm}{#1}\vcenter{\xymatrix@R=10pt@C=25pt{#2}}\end{equation}}
%
\newenvironment{mysubsection}[2][]
{\begin{subsec}\begin{upshape}\begin{bfseries}{#2.}
\end{bfseries}{#1}}
{\end{upshape}\end{subsec}}
\newenvironment{mysubsect}[2][]
{\begin{subsec}\begin{upshape}\begin{bfseries}{#2\vsn.}
\end{bfseries}{#1}}
{\end{upshape}\end{subsec}}
\newcommand{\supsect}[2]
{\vspace*{-5mm}\quad\\\begin{center}\textbf{{#1}}\vsm.~~~~\textbf{{#2}}\end{center}}
\newcommand{\sect}{\setcounter{figure}{0}\section}
%
%
\newcommand{\wh}{\ -- \ }
\newcommand{\wwh}{-- \ }
\newcommand{\w}[2][ ]{\ \ensuremath{#2}{#1}\ }
\newcommand{\ww}[1]{\ \ensuremath{#1}}
\newcommand{\www}[2][ ]{\ensuremath{#2}{#1}\ }
\newcommand{\wwb}[1]{\ \ensuremath{(#1)}-}
\newcommand{\wb}[2][ ]{\ (\ensuremath{#2}){#1}\ }
\newcommand{\wref}[2][ ]{\ (\ref{#2}){#1}\ }
\newcommand{\wwref}[2]{\ (\ref{#1})-(\ref{#2})\ }
%
%
\newcommand{\hsp}{\hspace*{9 mm}}
\newcommand{\hs}{\hspace*{4 mm}}
\newcommand{\hsn}{\hspace*{1 mm}}
\newcommand{\hsm}{\hspace*{2 mm}}
\newcommand{\vsn}{\vspace{2 mm}}
\newcommand{\vs}{\vspace{5 mm}}
\newcommand{\vsm}{\vspace{4 mm}}
%
%
\newcommand{\hra}{\hookrightarrow}
\newcommand{\xra}[1]{\xrightarrow{#1}}
\newcommand{\xepic}[1]{\xrightarrow{#1}\hspace{-5 mm}\to}
\newcommand{\lora}{\longrightarrow}
\newcommand{\lra}[1]{\langle{#1}\rangle}
\newcommand{\llrra}[1]{\langle\langle{#1}\rangle\rangle}
\newcommand{\llrr}[2]{\llrra{#1}\sb{#2}}
\newcommand{\llrrp}[2]{\llrra{#1}'\sb{#2}}
\newcommand{\lrf}{\langle\langle f\lo{0,1}\rangle\rangle}
\newcommand{\lrfn}[1]{\lrf\sb{#1}}
\newcommand{\lras}[1]{\langle{#1}\rangle\sb{\ast}}
\newcommand{\lrau}[1]{\langle{#1}\rangle\sp{\ast}}
\newcommand{\vlam}{\vec{\lambda}}
\newcommand{\epic}{\to\hspace{-3.5 mm}\to}
\newcommand{\xhra}[1]{\overset{#1}{\hookrightarrow}}
\newcommand{\efp}{\to\hspace{-1.5 mm}\rule{0.1mm}{2.2mm}\hspace{1.2mm}}
\newcommand{\efpic}{\mbox{$\to\hspace{-3.5 mm}\efp$}}
\newcommand{\up}[1]{\sp{(#1)}}
\newcommand{\bup}[1]{\sp{[{#1}]}}
\newcommand{\lo}[1]{\sb{(#1)}}
\newcommand{\bp}[1]{\sb{[#1]}}
\newcommand{\hfsm}[2]{{#1}\ltimes{#2}}
\newcommand{\sms}[2]{{#1}\wedge{#2}}
\newcommand{\rest}[1]{\lvert\sb{#1}}
%
%
\newcommand{\ab}{\operatorname{ab}}
\newcommand{\Aut}{\operatorname{Aut}}
\newcommand{\Cof}{\operatorname{Cof}}
\newcommand{\Coker}{\operatorname{Coker}}
\newcommand{\colim}{\operatorname{colim}}
\newcommand{\comp}{\mbox{\sf comp}}
\newcommand{\Cone}{\operatorname{Cone}}
\newcommand{\csk}[1]{\operatorname{csk}\sp{#1}}
\newcommand{\diag}{\operatorname{diag}}
\newcommand{\ev}{\operatorname{ev}}
\newcommand{\Ext}{\operatorname{Ext}}
\newcommand{\Fib}{\operatorname{Fib}}
\newcommand{\ho}{\operatorname{ho}}
\newcommand{\hocofib}{\operatorname{hocofib}}
\newcommand{\hocolim}{\operatorname{hocolim}}
\newcommand{\holim}{\operatorname{holim}}
\newcommand{\Hom}{\operatorname{Hom}}
\newcommand{\inc}{\operatorname{inc}}
\newcommand{\Ker}{\operatorname{Ker}}
\newcommand{\Id}{\operatorname{Id}}
\newcommand{\Image}{\operatorname{Im}}
\newcommand{\Obj}[1]{\operatorname{Obj}\,{#1}}
\newcommand{\op}{\sp{\operatorname{op}}}
\newcommand{\pt}{\operatorname{pt}}
\newcommand{\sk}[1]{\operatorname{sk}\sb{#1}}
\newcommand{\Sq}[1]{\operatorname{Sq}\sp{#1}}
\newcommand{\Tot}{\operatorname{Tot}}
\newcommand{\uTot}{\underline{\Tot}}
%
%
\newcommand{\map}{\operatorname{map}}
%
%
\newcommand{\Hu}[3]{H\sp{#1}({#2};{#3})}
\newcommand{\Hus}[2]{\Hu{\ast}{#1}{#2}}
\newcommand{\HiR}[1]{\Hu{\ast}{#1}{R}}
%
%
\newcommand{\A}{\mathcal{A}}
\newcommand{\B}{\mathcal{B}}
\newcommand{\C}{\mathcal{C}}
\newcommand{\D}{\mathcal{D}}
\newcommand{\E}{\mathcal{E}}
\newcommand{\F}{\mathcal{F}}
\newcommand{\M}{\mathcal{M}}
\newcommand{\Map}{{\EuScript Map}}
\newcommand{\MA}{\Map\sb{\A}}
\newcommand{\MAB}{\Map\sb{\A}\sp{\B}}
\newcommand{\MB}{\Map\sp{\B}}
\newcommand{\OO}{\mathcal{O}}
\newcommand{\PP}[2]{\mathcal{P}\sp{#1}\sb{#2}}
\newcommand{\hP}[2]{\widehat{\mathcal{P}}\sp{#1}\sb{#2}}
\newcommand{\Ss}{\mathcal{S}}
\newcommand{\Sa}{\Ss\sb{\ast}}
\newcommand{\U}{\mathcal{U}}
\newcommand{\eW}{{\EuScript W}}
\newcommand{\eX}{{\EuScript X}}
\newcommand{\eY}{{\EuScript Y}}
\newcommand{\eZ}{{\EuScript Z}}
%
%
\newcommand{\hy}[2]{{#1}\text{-}{#2}}
\newcommand{\Alg}[1]{{#1}\text{-}{\mbox{\sf Alg}}}
\newcommand{\Mod}[1]{{#1}\text{-}{\mbox{\sf Mod}}}
\newcommand{\Set}{\mbox{\sf Set}}
\newcommand{\Seta}{\Set\sb{\ast}}
\newcommand{\Cat}{\mbox{\sf Cat}}
\newcommand{\Ch}{\mbox{\sf Ch}}
\newcommand{\Grp}{\mbox{\sf Gp}}
\newcommand{\OC}{\hy{\OO}{\Cat}}
\newcommand{\SC}{\hy{\Ss}{\Cat}}
\newcommand{\SaC}{\hy{\Sa}{\Cat}}
\newcommand{\SO}{(\Ss,\OO)}
\newcommand{\SaO}{(\Sa,\OO)}
\newcommand{\SOC}{\hy{\SO}{\Cat}}
\newcommand{\SaOC}{\hy{\SaO}{\Cat}}
\newcommand{\Top}{\mbox{\sf Top}}
\newcommand{\Topa}{\Top\sb{\ast}}
%
%
\newcommand{\FF}{\mathbb F}
\newcommand{\Fp}{\FF\sb{p}}
\newcommand{\Ft}{\FF\sb{2}}
\newcommand{\Fq}{\FF\sb{q}}
\newcommand{\NN}{\mathbb N}
\newcommand{\QQ}{\mathbb Q}
\newcommand{\ZZ}{\mathbb Z}
%
%
\newcommand{\bA}{{\mathbf A}}
\newcommand{\bB}{{\mathbf B}}
\newcommand{\bC}{{\mathbf C}}
\newcommand{\bD}{{\mathbf D}}
\newcommand{\cD}{D}
\newcommand{\bDel}{\mathbf{\Delta}}
\newcommand{\Del}[1]{\bDel\sp{#1}}
\newcommand{\Dels}[2]{\Del{#1}\lo{#2}}
\newcommand{\Dlt}[1]{\Delta[{#1}]}
\newcommand{\bE}{{\mathbf E}}
\newcommand{\be}[1]{{\mathbf e}\sp{#1}}
\newcommand{\bF}{{\mathbf F}}
\newcommand{\Fv}[1]{\bF\bp{#1}}
\newcommand{\Fk}[1]{F\sp{#1}}
\newcommand{\Fpp}{\hspace{0.5mm}'\hspace{-0.7mm}F}
\newcommand{\Fpk}[1]{\hspace{0.5mm}'\hspace{-0.7mm}F\sp{#1}}
\newcommand{\Fn}[2]{F\sp{#1}\bp{#2}}
\newcommand{\bG}{{\mathbf G}}
\newcommand{\Gv}[1]{\bG\bp{#1}}
\newcommand{\hGv}[1]{\widehat{\bG}\bp{#1}}
\newcommand{\Gn}[2]{G\sp{#1}\bp{#2}}
\newcommand{\hGn}[2]{\widehat{G}\sp{#1}\bp{#2}}
\newcommand{\tg}[1]{\widetilde{g}\sp{#1}}
\newcommand{\bH}{{\mathbf H}}
\newcommand{\wH}{\widehat{H}}
\newcommand{\Hv}[1]{\bH\bp{#1}}
\newcommand{\hHv}[1]{\wH\bp{#1}}
\newcommand{\Hn}[2]{H\sp{#1}\bp{#2}}
\newcommand{\tH}[1]{\widetilde{H}\sp{#1}}
\newcommand{\tHn}[2]{\tH{#1}\bp{#2}}
\newcommand{\bK}{{\mathbf K}}
\newcommand{\KP}[2]{\bK({#1},{#2})}
\newcommand{\KR}[1]{\KP{R}{#1}}
\newcommand{\KZ}[1]{\KP{\ZZ}{#1}}
\newcommand{\KF}[1]{\KP{\Fp}{#1}}
\newcommand{\bM}{{\mathbf M}}
\newcommand{\bS}[1]{{\mathbf S}\sp{#1}}
\newcommand{\bT}{\mathbf{\Theta}}
\newcommand{\TA}{\bT\sp{\A}}
\newcommand{\TB}{\bT\sb{\B}}
\newcommand{\ThB}{\Theta\sb{\B}}
\newcommand{\TAB}{\bT\sp{\A}\sb{\B}}
\newcommand{\TR}{\Theta\sb{R}}
\newcommand{\TRl}{\Theta\sb{R}\sp{\lambda}}
\newcommand{\bU}{{\mathbf U}}
\newcommand{\bV}{{\mathbf V}}
\newcommand{\bW}{{\mathbf W}}
\newcommand{\bX}{{\mathbf X}}
\newcommand{\bY}{{\mathbf Y}}
\newcommand{\hY}{\widehat{\bY}}
\newcommand{\bZ}{{\mathbf Z}}
%
%
\newcommand{\cu}[1]{c({#1})\sp{\bullet}}
\newcommand{\cd}[1]{c({#1})\sb{\bullet}}
\newcommand{\vare}{\varepsilon}
\newcommand{\bv}{{\bm \vare}}
\newcommand{\bve}[1]{\bv\sb{[{#1}]}}
\newcommand{\bk}{[\mathbf{k}]}
\newcommand{\bn}{[\mathbf{n}]}
\newcommand{\dz}[1]{d\sp{0}\sb{#1}}
\newcommand{\od}{\overline{\partial}}
\newcommand{\ud}{\overline{d}}
\newcommand{\vdz}[1]{{\underline{d}\sp{0}\sb{#1}}}
\newcommand{\odz}[1]{\od\sp{#1}\sb{0}}
\newcommand{\udz}[1]{\ud\sp{0}\sb{#1}}
\newcommand{\Ad}{A\sb{\bullet}}
\newcommand{\Au}{A\sp{\bullet}}
\newcommand{\Bu}{B\sp{\bullet}}
\newcommand{\Bd}{B\sb{\bullet}}
\newcommand{\Cu}{\cM{\bullet}}
\newcommand{\Cd}{C\sb{\bullet}}
\newcommand{\Ds}{\cD\sp{\ast}}
\newcommand{\Du}[1]{\bD\sp{\bullet}\bp{#1}}
\newcommand{\oD}[1]{\overline{D}\sp{#1}}
\newcommand{\sD}[2]{\mathbf{\Sigma}\bD\sp{#1}\bp{#2}}
\newcommand{\sDu}[1]{\sD{\bullet}{#1}}
\newcommand{\Deu}{\Del{\bullet}}
\newcommand{\hM}[2]{\widehat{\bM}\sp{#1}\bp{#2}}
\newcommand{\tS}{\widetilde{\Sigma}}
\newcommand{\Ud}{\bU\sb{\bullet}}
\newcommand{\oV}[1]{\overline{V}\sb{#1}}
\newcommand{\Vd}{V\sb{\bullet}}
\newcommand{\Vu}{V\sp{\bullet}}
\newcommand{\oW}[1]{\overline{\bW}\hspace{0.3mm}\sp{#1}}
\newcommand{\Wu}{\bW\sp{\bullet}}
\newcommand{\Wn}[2]{\bW\sp{#1}\bp{#2}}
\newcommand{\W}[1]{\Wn{\bullet}{#1}}
\newcommand{\hW}{\widehat{\bW}}
\newcommand{\hWn}[2]{\hW\quad\hspace*{-4mm}\sp{#1}\bp{#2}}
\newcommand{\hWu}[1]{\hWn{\bullet}{#1}}
\newcommand{\tW}{\widetilde{\bW}}
\newcommand{\tWn}[2]{\tW\quad\hspace*{-4mm}\sp{#1}\bp{#2}}
\newcommand{\tWu}[1]{\tWn{\bullet}{#1}}
\newcommand{\prn}[1]{\pi\sb{[{#1}]}}
\newcommand{\Xu}{\bX\sp{\bullet}}
\newcommand{\Yu}{\bY\sp{\bullet}}
\newcommand{\Zu}{\bZ\sp{\bullet}}
%
%
\newcommand{\Bal}[1][ ]{$\ThB$-algebra{#1}}
\newcommand{\Tal}[1][ ]{$\Theta$-algebra{#1}}
\newcommand{\TRal}[1][ ]{$\TR$-algebra{#1}}
\newcommand{\TAlg}{\Alg{\Theta}}
\newcommand{\BAlg}{\Alg{\ThB}}
\newcommand{\RAlg}{\Alg{\TR}}
%
%
\newcommand{\cH}[3]{H\sp{#1}({#2};{#3})}
\newcommand{\HR}[2]{\cH{#1}{#2}{R}}
\newcommand{\HsR}[1]{\HR{\ast}{#1}}
\newcommand{\HF}[2]{\cH{#1}{#2}{\Fp}}
\newcommand{\HsF}[1]{\HF{\ast}{#1}}
%
%
\newcommand{\ma}[1][ ]{mapping algebra{#1}}
\newcommand{\Ama}[1][ ]{$\A$-mapping algebra{#1}}
\newcommand{\ABma}[1][ ]{$\A$-$\B$-bimapping algebra{#1}}
\newcommand{\Bma}[1][ ]{$\B$-dual mapping algebra{#1}}
\newcommand{\Tma}[1][ ]{$\bT$-mapping algebra{#1}}
\newcommand{\lin}[1]{\{{#1}\}}
\newcommand{\fff}{\mathfrak{f}}
\newcommand{\fM}{\mathfrak{M}}
\newcommand{\fMA}{\fM\sp{\A}}
\newcommand{\fMB}{\fM\sb{\B}}
\newcommand{\fMAB}{\fM\sp{\A}\sb{\B}}
\newcommand{\fMT}{\fM\sb{\bT}}
\newcommand{\fMR}{\fM\sb{R}}
\newcommand{\fT}{\mathfrak{T}}
\newcommand{\fTd}{\fT\sb{\bullet}}
\newcommand{\fV}{\mathfrak{V}}
\newcommand{\fVd}{\fV\sb{\bullet}}
\newcommand{\fVu}{\fV\sp{\bullet}}
\newcommand{\fW}{\mathfrak{W}}
\newcommand{\fWd}{\fW\sb{\bullet}}
\newcommand{\fX}{\mathfrak{X}}
\newcommand{\fY}{\mathfrak{Y}}
\newcommand{\fZ}{\mathfrak{Z}}
\newcommand{\fin}{\operatorname{fin}}
\newcommand{\init}{\operatorname{init}}
\newcommand{\vf}{v\sb{\fin}}
\newcommand{\vi}{v\sb{\init}}
%
%
%
\title{Higher structure in the unstable Adams spectral sequence}
%
\author{Samik Basu, David Blanc, and Debasis Sen}
\address{Department of Mathematical and Computational Science\\
Indian Association for the Cultivation of Science\\ Kolkata - 700032, India}
\email{mcssb@iacs.res.in}
\address{Department of Mathematics\\ University of Haifa\\ 3498838 Haifa\\ Israel}
\email{blanc@math.haifa.ac.il}
\address{Department of Mathematics and Statistics\\
Indian Institute of Technology, Kanpur\\ Uttar Pradesh 208016\\ India}
\email{debasis@iitk.ac.in}

\date{\today}

\subjclass[2010]{Primary: 55T15; \ secondary: 55S20, 55P60}

\keywords{Unstable Adams spectral sequence, higher cohomology operations, 
differentials, cosimplicial resolutions}

\begin{abstract}
We describe a variant construction of the unstable Adams spectral 
sequence for a space $\bY$, associated to any free simplicial resolution of
\w[,]{\HiR{\bY}} for \w{R=\Fp} or $\QQ$. We use this construction to describe the 
differentials and filtration in the spectral sequence in terms of appropriate 
systems of  higher cohomology operations.
\end{abstract}

\maketitle

\setcounter{section}{0}

%
%
\section*{Introduction}
\label{cint}

The original Adams spectral sequence of \cite{AdSS} calculates the stable
homotopy groups of a space $\bY$ at a prime $p$, starting with its
\ww{\Fp}-cohomology. Later, several unstable versions of this were proposed
(see \cite{CurtR,RectU,MPetM,BCurS,BKanS}), all shown in \cite[X,\S 6]{BKanH} to
agree for reasonable spaces $\bY$. There are also variants for computing
\w[,]{\pi\sb{\ast}\map(\bX,\bY)} as well as for more general coefficients,
but for simplicity we restrict attention here to the original version, for
coefficients in \w{R=\Fp} or $\QQ$.

The \ $E\sb{2}$-terms of both the stable and unstable spectral sequences for $\bY$
can be identified as certain graded \ww{\Ext} groups associated to \w[,]{\HiR{\bY}}
equipped with an action of the (stable or unstable) primary $R$-cohomology operations
(cf.\ \cite{AdSS,BKanS}). These can be computed from any resolution \w{\Vd}
of \w[,]{\HiR{\bY}} in an appropriate category of \TRal[s] (for \w[:]{R=\Fp} these
are modules or unstable algebras, respectively, over the mod $p$ Steenrod algebra).

We show here how, as in the stable case, the unstable Adams spectral
sequence for $\bY$ can be obtained from a realization
of any such algebraic resolution \w{\Vd\to\HiR{\bY}} by a cosimplicial space 
\w[,]{\Wu} constructed inductively through successive approximations \w{\W{n}}
which we think of as forming an unstable Adams resolution of $\bY$.

\begin{mysubsect}{Systems of higher cohomology operations}\label{sshoo}

In \cite{BSenH}, we showed how this construction of \w{\W{n}} can be used to define
certain ``universal higher cohomology operations'' associated to each $R$-good
space $\bY$, which can be used to distinguish it from other space $\bZ$ having
\w{\HiR{\bZ}\cong\HiR{\bY}} (as \TRal[s).] Similar higher operations can also be used
to distinguish among homotopy classes of maps \w{f\sb{0},f\sb{1}:\bX\to\bY} between
$R$-good spaces which induce the same map
\w{f\sb{0}\sp{\ast}=f\sb{1}\sp{\ast}:\HiR{\bY}\to\HiR{\bX}} in cohomology.

Our second goal here is to show that analogous higher operations define the
differentials in the unstable Adams spectral sequence for $\bY$, as well as the
filtration index of each element in \w[.]{\pi\sb{\ast}\bY}

The notion of secondary operations in homotopy theory goes back at least to
the 1950's, when Massey products, Toda brackets, and Adem's secondary cohomology
operations first appeared (in \cite{MassN,TodG,TodC,AdemI}). Since then, there have
been several attempts to give general definitions of higher order operations (see
\cite{SpanH,MaunC,GWalkL,BMarkH,BJTurnH}), but none have been completely satisfactory.

Rather than trying to give one definition covering all variants, we will describe the
basic properties we expect of a general $n$-th order homotopy operation
\w{\llrr{X}{n}} for \w[:]{n\geq 2}

\begin{enumerate}
\renewcommand{\labelenumi}{(\alph{enumi})~}
\item It serves as the final obstruction to rectifying a homotopy-commutative
diagram \w{X:I\to\ho\M} \wwh or equivalently, making it $\infty$-homotopy
commutative \wh where $\M$ is a pointed simplicial model category and
$I$ is a suitable finite directed diagram of length \w[.]{n+1}
\item Its \emph{value}, for an appropriate choice of \emph{initial data}
\w[,]{G\up{n-1}} is a homotopy class in \w[,]{[X(\vi),\,\Omega\sp{n-1}X(\vf)]}
where \w{\vi} is weakly initial in $I$ and \w{\vf} is weakly terminal (cf.\
\cite[\S 2.1]{BMarkH}). This value is zero (the operation \emph{vanishes}) if and only
if the diagram can be rectified for this choice of initial data.
\item It has an associated system of lower order operations
\w[,]{(\llrr{X\rest{I\sb{k}}}{k})\sb{k=2}\sp{n-1}} corresponding to the filtration
of $I$ by initial (or final) segments \w{I\sb{k}} of length \w[.]{k+1}
The initial data \w{G\up{k}} for \w{\llrr{X}{k+1}} is determined by a
rectification of \w{X\rest{I\sb{k}}} made fibrant or cofibrant (in an appropriate
model category structure on \w[)]{\M\sp{I\sb{k}}} \wwh thus assuming in particular
that \w{\llrr{X\rest{I\sb{k}}}{k}} vanishes.
\item Two $n$-th order operations (for different indexing categories $I$ and $J$) are
\emph{equivalent} if the corresponding rectification problems are equivalent \wh
so there is a bijective correspondence of the initial data for the two, and the
resulting value for one vanishes if and only if it does so for the other.

We say that they are \emph{strongly equivalent} if the correspondence induces a
bijection of values in \w{[X(\vi),\,\Omega\sp{n-1}X(\vf)]} (so in particular the
two diagrams have the same initial and final objects in $\M$).
\item When $\M$ is a category of spaces or spectra, we say that \w{\llrr{X}{n}}
is an $n$-the order $R$-\emph{cohomology} operation, and that
\w{(\llrr{X\rest{I\sb{k}}}{k})\sb{k=2}\sp{n}} is a
\emph{system of higher $R$-cohomology operations}, if \w{Y(v)} is an $R$-GEM
for all \w[,]{v\in\Obj(J)\setminus\{\vi\}} for some strongly equivalent system
associated to \w[.]{Y:J\to\ho\M} 

Note that the spaces in the chosen diagram \w{X:I\to\ho\M} itself need not all 
be $R$-GEMs \wh only those in some equivalent system. See last paragraph in 
\S \ref{shco} below.
\end{enumerate}
\end{mysubsect}

\begin{mysubsect}{Main results}\label{smr}

This paper continues the project begun in \cite{BSenH}, intended to show how  
higher cohomology operations serve as a unifying setting for describing finer 
homotopy invariants of (the $R$-completion of) topological spaces.

This principle is applied here to the unstable Adams spectral sequence, 
but in fact Theorems B and C below apply equally to the stable Adams spectral 
sequence (since the former agrees with the latter in the stable range).

In Section \ref{crcwr} we recall from \cite{BSenH} how to associate to any CW 
resolution \w{\Vd} of the \TRal \w{\HiR{\bY}} (for any $R$-good space $\bY$),
an \emph{unstable Adams resolution}: that is to say, a cosimplicial space 
\w[,]{\Wu} obtained as the limit of a tower of fibrations:
$$
(\star)\hspace{20mm}
\dotsc~\to~\W{n}~\xra{\prn{n}}~\W{n-1}~\xra{\prn{n-1}}~\W{n-2}~\to~\dotsc~
\to~\W{0}~.\hspace{20mm}\quad
$$
\noindent Each stage \w{\W{n}} in this tower realizes the corresponding
skeleton \w{\sk{n}\Vd} of the algebraic resolution \w[;]{\Vd} this in turn
is obtained from \w{\sk{n-1}\Vd} by attaching a free \TRal \w{\oV{n}} by a suitable
map (as in the usual construction of a CW complex). We can realize \w{\oV{n}} by an
$R$-GEM \w[.]{\oW{n}}

In Section \ref{cuass} we then show:

\begin{thma}
The homotopy spectral sequence for the tower \w{(\star)} coincides with the 
usual unstable Adams spectral sequence for $\bY$.
\end{thma}
\noindent See Theorem \ref{tuass}.

In Section \ref{cduass}, we associate to any such unstable Adams resolution 
\w{\Wu} of $\bY$ a sequence of higher cohomology operations
\w{\llrr{-}{r}~:~\pi\sb{k+n}\oW{n}\to\pi\sb{k+n+r-1}\oW{n+r}} (see Definition
\ref{dhco}), and show:

\begin{thmb}
Each value \w{\llrr{\gamma}{r}\in\pi\sb{k+n+r-1}\oW{n+r}=E\sb{1}\sp{n+r,k+n+r-1}}
of the $r$-th order operation \w{\llrr{-}{r}} represents the result of applying
the differential \w{d\sb{r}} to the element of \w{E\sb{r}\sp{n,k+n}} represented by
\w[.]{\gamma\in\pi\sb{k+n}\oW{n}=E\sb{1}\sp{n,k+n}}
\end{thmb}
\noindent See Theorem \ref{thco}.

Finally, in Section \ref{cfuass} we produce another sequence of higher cohomology 
operations \w[,]{\llrrp{-}{r}:\pi\sb{k}\bY\to\pi\sb{k+n}\oW{n}} 
and prove

\begin{thmc}
For any \w[,]{0\neq\gamma\in\pi\sb{k}\bY} the operation \w{\llrrp{\gamma}{n-1}}
vanishes while \w{\llrrp{\gamma}{n}\neq 0} if and only if $\gamma$ is represented
in the unstable Adams spectral sequence in filtration $n$ by the value of
\w[.]{\llrrp{\gamma}{n}\in\pi\sb{k}\Omega\sp{n}\oW{n}}
\end{thmc}
\noindent See Theorem \ref{tfil}.

A simple example of the secondary cohomology operation associated to
an element in Adams filtration $1$ is given in \S \ref{eqsecopn}.
\end{mysubsect}

\begin{notation}\label{snac}
The category of finite ordered sets and order-preserving maps will be denoted
by $\Delta$ (cf.\ \cite[\S 2]{MayS}), with objects \w{\bn=[0<1<\dotsc n]}
\wb[,]{n\in\NN} so a \emph{cosimplicial object} \w{\Au}
in a category $\C$ is a functor \w[,]{\Delta\to\C} and a \emph{simplicial object}
\w{\Ad} in $\C$ is a functor \w[.]{\Delta\op\to\C} Write \w{c\C=\C\sp{\Delta}} for
the category of cosimplicial objects in $\C$, and \w{s\C=\C\sp{\Delta\op}} for
that of simplicial objects. There is a natural embedding \w{\cu{-}:\C\to c\C}
(the constant cosimplicial object), and similarly \w[.]{\cd{-}:\C\to s\C}

If \w{\Delta\sb{+}} denotes the subcategory of injective
maps in $\Delta$, a functor \w{\Delta\sb{+}\to\C} will be called a
\emph{restricted cosimplicial object}.

A \emph{chain complex} in a pointed category $\C$ is a sequence of maps
\w{\partial\sb{n}:A\sb{n}\to A\sb{n-1}} with
\w{\partial\sb{n}\circ\partial\sb{n+1}=0} for each \w[.]{n\geq 1} The category of
such will be denoted by \w[.]{\Ch\sb{\C}} The category of \emph{cochain complexes}
in $\C$, defined dually, is denoted by \w[.]{\Ch\sp{\C}}

The category of simplicial sets will be denoted by
\w[,]{\Ss=s\Set} and that of pointed simplicial sets (here called simply ``spaces'')
by \w[.]{\Sa=s\Seta}
Write \w{\map\sb{\C}(\bX,\bY)} for the standard function complex in a simplicial
model category $\C$  (see \cite[I, \S 1.5]{GJarS}).

The \emph{half-smash} of $\bX$ and $\bY$, where \w{(\bY,y)\in\Sa} is pointed, but
\w{\bX\in\Ss} is not, is denoted by
\w[.]{\hfsm{\bX}{\bY}:=(\bX\times\bY)/(\bX\times\{y\})\in\Sa}
In particular, the (unreduced) \emph{cone} on $\bX$ is \w[,]{C\bX:=\hfsm{\bX}{[0,1]}}
where the interval \w{[0,1]} has base point $1$, while the \emph{reduced cone}
on $\bY$ is \w[.]{\bar{C}\bY:=\sms{\bY}{[0,1]}}
\end{notation}

\begin{ack}
The research of the first author was partially supported by NBHM project 
3743, that of the second author by Israel Science
Foundation grant 770/16, and that of the third author by INSPIRE grant No.~IFA MA-12.
\end{ack}

%
%
\sect{Background}
\label{cback}

We first recall some background material on cohomology algebras, their
resolutions, and the realizations of these resolutions.

\begin{defn}\label{drth}
For any ring $R$ and limit cardinal $\lambda$, let \w{\TR=\TRl} denote (a skeleton
of) the full subcategory of \w{\ho\Sa} spanned by finite products of objects of
the form \w[,]{\{\KP{V}{n}\,:\,n\in\NN\sb{>0}\}} where $V$ is a an $R$-module
generated by a set of cardinality $<\lambda$,
This is a (multi-sorted) \emph{theory}, in the sense of Lawvere
(see \cite{LawvF} and \cite{EhreET}). A product-preserving functor
\w{\Gamma:\TR\to\Seta} will be called a \emph{\TRal} (cf.\ \cite[\S 5.6]{BorcH2}).
Since each \w{\bB\in\TR} is an $R$-module object, all \TRal[s] take values in
$R$-modules, and their category will be denoted by \w[.]{\RAlg}

In particular, a \TRal $\Gamma$ is \emph{realizable} if it is represented by a space
\w[,]{\bY\in\Sa} with \w[.]{\Gamma\lin{\KR{n}}:=[\bY,\,\KR{n}]}
By abuse of notation, we denote such a $\Gamma$ by \w[.]{\HiR{\bY}}
Thus \w{\HiR{\bY}} is just the $R$-cohomology algebra of $\bY$, equipped with
the action of the primary $R$-cohomology operations.

A \TRal of the form \w{\HiR{\bB}} for \w{\bB\in\TRl} is called
\emph{free}.  Note that this definition depends on our choice of cardinal
$\lambda$ (cf.\ \cite[\S 1.25]{BSenH}).
\end{defn}

\begin{example}\label{egrtal}
If \w{\lambda=\aleph\sb{0}} and \w[,]{R=\Fp} \w{\TR} consists of finite
$R$-GEMs, and a \TRal is an unstable algebra over the mod $p$ Steenrod algebra,
as in \cite[\S 1.4]{SchwU}. When \w[,]{R=\QQ} a \TRal is just a
graded-commutative $\QQ$-algebra.
\end{example}

\begin{mysubsection}{Algebraic resolutions}\label{sares}
As in \cite[II, \S 4]{QuiH}, there is a model category structure on the category
\w{s\RAlg} of simplicial \TRal[s,] so there is a notion of a free simplicial
resolution \w{\Vd} of a \TRal $\Gamma$.

We shall be interested in a particular kind, known as a \emph{CW-resolution}
(cf.\ \cite[\S 3.10]{BlaCW}), defined as follows
\end{mysubsection}

\begin{defn}\label{dmoore}
Recall that for any simplicial object \w{\Vd} over a complete
pointed category $\M$, the $n$-th \emph{Moore chains} object \wb{n\geq 0} is
$$
C\sb{n}\Vd~:=~\bigcap\sb{i=1}\sp{k}\Ker(d\sb{i})
$$
\noindent with differential \w{\partial\sb{n}:=d\sb{0}} satisfying
\w[.]{\partial\sb{n}\circ\partial\sb{n+1}=0} The $n$-th \emph{Moore cycles}
object is \w[.]{Z\sb{n}\Vd:=\Ker(\partial\sb{n})}

If $\M$ is cocomplete, the $n$-th \emph{latching object} \w{L\sb{n}\Vd} is
defined to be \w{\colim\sb{\theta:\bk\to\bn}\,V\sb{k}} (cf.\ \S \ref{snac}),
equipped with the obvious canonical map \w{\delta:V\sb{n}\to L\sb{n}\Vd}
(see \cite[VII, \S 1]{GJarS}).
\end{defn}

\begin{defn}\label{dcwres}
We say that \w{\Vd\in s\M} is a \emph{CW object} if it is equipped with a
\emph{CW basis} \w{(\oV{n})\sb{n=0}\sp{\infty}} in $\M$ such that
\w[,]{V\sb{n}=\oV{n}\amalg L\sb{n}\Vd} and \w{d\sb{i}\rest{\oV{n}}=0}
for \w[.]{1\leq i\leq n} In this case
\w{\odz{V\sb{n}}:=d\sb{0}\rest{\oV{n}}:\oV{n}\to V\sb{n-1}} is called the
attaching map for \w[.]{\oV{n}} By the simplicial identities \w{\odz{V\sb{n}}}
factors as \w[.]{\odz{V\sb{n}}:\oV{n}\to Z\sb{n-1}\Vd\subset V\sb{n-1}}

In this case we have an explicit description
\begin{myeq}\label{eqslatch}
L\sb{n}\Vd~:=~
\coprod\sb{0\leq k\leq n}~\coprod\sb{0\leq i\sb{1}<\dotsc<i\sb{n-k-1}\leq n-1}~
\oV{k}
\end{myeq}
\noindent for its $n$-th latching object, in which the iterated degeneracy map
\w[,]{s\sb{i\sb{n-k-1}}\dotsc s\sb{i\sb{2}}s\sb{i\sb{1}}} restricted to the
basis \w[,]{\oV{k}} is the inclusion into the copy of \w{\oV{k}} indexed by
\w[.]{(i\sb{1},\dotsc,i\sb{n-k-1})}

In particular, if in \w{\M=\RAlg} we set \w{Z\sb{-1}\Vd:=\Gamma\in\M} and require
that \w{\oV{n}} be free and that \w{\odz{V\sb{n}}:\oV{n}\epic Z\sb{n-1}\Vd}
be surjective for each \w[,]{n\geq 0} we call the resulting augmented free
simplicial \TRal \w{\Vd\to\Gamma} a \emph{CW resolution}.
\end{defn}

\begin{defn}\label{dmoorc}
Dually, for a cosimplicial object \w{\Vu} over a cocomplete
pointed category $\M$ and \w[,]{n\geq 0} the $n$-th \emph{Moore cochains}
object is
$$
C\sp{n}\Vu~:=~\Cof\left(\coprod\sb{i=1}\sp{n-1}\,V\sp{n-1}~\xra{\bot\sb{i}\,d\sp{i}}~
V\sp{n-1}\right)~,
$$
\noindent with differential \w{\delta\sp{n-1}:C\sp{n-1}\Vu\to C\sp{n}\Vu}
induced by \w[,]{d\sp{0}\sb{n-1}} and structure map
\w[.]{v\sp{n}:V\sp{n}\to C\sp{n}\Vu} We denote the cofiber of
\w{\delta\sp{n-1}} by \w[,]{Z\sp{n}\Vu}
with structure map \w[.]{w\sp{n}:C\sp{n}\Vu\to Z\sp{n}\Vu}

If $\M$ is complete, the $n$-th \emph{matching object} for \w{\Vu} is
\begin{myeq}\label{eqmatch}
M\sp{n}\Vu~:=~\lim\sb{\phi:\bn\to \bk}\,V\sp{k}~,
\end{myeq}
\noindent where $\phi$ ranges over the surjective maps \w{\bn\to\bk} in
$\Delta$, with the obvious natural map \w{\zeta\sp{n}:V\sp{n}\to M\sp{n-1}\Vu}
(through which all codegeneracies factor).
If $\M$ is a model category, we say \w{\Vu} is \emph{(Reedy) fibrant} if each
map \w{\zeta\sp{n}} is a fibration (see \cite[X, \S 4]{BKanH}).
\end{defn}

\begin{mysubsection}{Cosimplicial resolutions}\label{scosimp}
Let \w{\Wu} be a weak $R$-resolution of $\bY$ (see \cite[\S 6.1]{BousC}) \wh that
is, a cosimplicial space with each \w{\bW\sp{n}} an $R$-GEM,
equipped with a coaugmentation \w{\vare:\bY\to\Wu} which is an $R$-equivalence
(cf.\ \cite[\S 3.2]{BousC}). We assume for simplicity that \w{\Wu} is Reedy
cofibrant(see \cite[\S 15.3]{PHirM}, so \w[,]{\Tot\Wu\simeq\hY} the $R$-completion
of $\bY$, by \cite[Theorem 6.5]{BousC}.
Such a resolution may be constructed functorially using a suitable monad, as in
\cite[I, \S 2]{BKanH} (see also \cite[\S 2-3]{BSenM})
\end{mysubsection}

%
%
\sect{Realizing CW resolutions}
\label{crcwr}

Our main technical tool in this paper is the construction of appropriate
cosimplicial resolutions of an ($R$-good) space $\bY$, realizing a given
algebraic resolution of \w[.]{\HsR{\bY}} These are the unstable analogue of
the Adams resolution of a space or spectrum (see, e.g., \cite[\S 2.2]{RavC}).

\begin{mysubsect}{Cosimplicial CW resolutions}\label{srcwr}

In \cite[\S 2]{BSenH} we showed that, given a space $\bY$ with
\w[,]{\Gamma:=\HsR{\bY}} any CW resolution \w{\Vd\in s\RAlg} of $\Gamma$
with CW basis \w{(\oV{n})\sb{n=0}\sp{\infty}} (\S \ref{dcwres}) can be realized
by a coaugmented cosimplicial space \w[.]{\bY\to\Wu} This \w{\Wu} is the limit of
a tower of Reedy fibrant and cofibrant $\bY$-coaugmented cosimplicial spaces:
\begin{myeq}\label{eqtower}
\dotsc~\to~\W{n}~\xra{\prn{n}}~\W{n-1}~\xra{\prn{n-1}}~\W{n-2}~\to~\dotsc~
\to~\W{0}~,
\end{myeq}
\noindent in \w[,]{c\Sa=\Sa\sp{\Delta}} with each \w{\prn{n}} a Reedy fibration.

The passage from \w{\W{n-1}} to \w{\W{n}} is as follows:

\begin{enumerate}
\renewcommand{\labelenumi}{(\alph{enumi})~}
\item Choose an $R$-GEM \w{\oW{n}} realizing the free \TRal \w{\oV{n}} (this
is possible because of our choice of $\lambda$ in \S \ref{drth}).
\item The $n$-th attaching map \w{\odz{}:\oV{n}\to C\sb{n-1}\Vu}
defines a unique map \w{\phi:\oV{n}\otimes S\sp{n-1}\to C\sb{\ast}\sk{n-1}\Vd}
of chain complexes in \w[,]{\RAlg}  where \w{\oV{}\otimes S\sp{n-1}} is the
chain complex with $\oV{}$ in dimension \w[,]{n-1} and $0$ elsewhere.

Evidently, one can realize \w{\oV{}\otimes S\sp{n-1}} by a cochain complex in
\w[;]{\Sa} we choose a realization \w{\Ds} which is a Reedy
fibrant cochain complex in \w{\Sa} in the sense of \cite[\S 2.4(i)]{BSenH},
by setting
\begin{myeq}\label{eqdk}
D\sp{k}~:=~P\Omega\sp{n-k-2}\oW{n}
\end{myeq}
\noindent for each \w[,]{k\geq 0}
where \w{P\Omega\sp{-1}\oW{n}:=\oW{n}} and \w{P\Omega\sp{k}\oW{n}:=\ast}
for \w[.]{k<-1} The differential is \w[,]{\iota\circ p} where \w{p:P\bX\to\bX}
is the appropriate path fibration and \w{\iota:\Omega\bX\to P\bX} is the inclusion.
\item Note that the Moore cochains define a functor
\w{C\sp{\ast}:\Sa\sp{\Delta\sb{+}}\to\Ch\sp{\Sa}} into the category of cochain
complexes of spaces (\S \ref{snac}), with right adjoint $\E$, so if we can realize
$\phi$ by a cochain map \w{\Fpp:C\sp{\ast}\W{n-1}\to\Ds} (see Proposition
\ref{pvanish} below), it induces \w{\widetilde{F}:\U\W{n-1}\to\E\Ds} (where
\w{\U:\Sa\sp{\Delta}\to\Sa\sp{\Delta\sb{+}}} is the forgetful functor \wh
see \S \ref{snac})

Taking the cofiber of $\widetilde{F}$ in \w{\Sa\sp{\Delta\sb{+}}} yields
a restricted cosimplicial space \w{\tWu{n}} with
\begin{myeq}\label{eqdopb}
\tWn{r}{n}~=~\Wn{r}{n-1}\times P\Omega\sp{n-r-1}\oW{n}~.
\end{myeq}
\noindent
\item We add the missing codegeneracies form a full cosimplicial space
\w[,]{\hWu{n}} as follows: Set
\w[,]{M\sp{r}\hWu{n}~:=~M\sp{r}\W{n-1}~\times\hM{r}{n}} where
\begin{myeq}\label{eqsmatch}
\hM{r}{n}~:=~\prod\sb{0\leq k\leq r}\ \prod\sb{0\leq i\sb{1}<\dotsc<i\sb{k}\leq r}
P\Omega\sp{n+k-r-1}\oW{n}
\end{myeq}
\noindent for each \w[.]{r\geq 0}  We then set
\begin{myeq}\label{eqcindstep}
\begin{split}
\hWn{r}{n}~:=&~\tWn{r}{n}~\times~\hM{r-1}{n}~=~
\Wn{r}{n-1}~\times~\hM{r-1}{n}~\times~P\Omega\sp{n-r-1}\oW{n}\\
~=&~\Wn{r}{n-1}~\times
\prod\sb{0\leq k\leq r}\ \prod\sb{0\leq i\sb{1}<\dotsc<i\sb{k}\leq r-1}
P\Omega\sp{n+k-r}\oW{n}~,
\end{split}
\end{myeq}
\noindent and the codegeneracy map \w{s\sp{t}:\hWn{r+1}{n}\to\hWn{r}{n}} is defined
into the factor \w{Q:=P\Omega\sp{n+k-r-1}\oW{n}} of \w{\hWn{r}{n}} indexed by
the $k$-tuple \w{I=(i\sb{1},\dotsc,i\sb{k})} by projecting \w{\hWn{r+1}{n}} onto
the copy of $Q$ indexed by the unique
\wwb{k+1}tuple \w{J=(j\sb{1},\dotsc,j\sb{k+1})} satisfying the cosimplicial
identity \w[.]{s\sp{I}\circ s\sp{t}=s\sp{J}}

This defines a functor \w{\F:\Sa\sp{\Delta\sb{+}}\to\Sa\sp{\Delta}} (``add
codegeneracies''), with \w[,]{\F(\tWu{n}):=\hWu{n}}
left adjoint to \w[.]{\U:\Sa\sp{\Delta}\to\Sa\sp{\Delta\sb{+}}}
By adjunction we therefore have a map
\begin{myeq}\label{eqfwd}
\bF\bp{n-1}~:~\W{n-1}~\to~\F\E\Ds~=:~\Du{n}
\end{myeq}
\noindent determined by \w{\Fpp:C\sp{\ast}\W{n-1}\to\Ds} and the codegeneracies.

Note that (assuming the objects \w{\oW{n}} are all fibrant) the
cosimplicial space \w{\hWu{n}} we have constructed is Reedy fibrant, and
from \wref{eqcindstep} we see that the dimensionwise projection defines a
Reedy fibration \w[.]{r\bp{n}:\hWu{n}\to\W{n-1}}
We write \w{i\bp{n}:\sDu{n}\hra\hWu{n}} for the inclusion of the fiber
\w{\sDu{n}:=\F\E\tS\Ds} of \w[.]{r\bp{n}}

Here \w{\tS\Ds} is the obvious Reedy fibrant cochain complex in \w{\Sa} realizing
\w[.]{\oV{}\otimes S\sp{n}}  Note that the unique non-zero coface map into
the non-codegenerate part \w{P\Omega\sp{n-k-1}\oW{n}} of \w{\sD{k}{n}} is
\w[,]{d\sp{1}} not \w[.]{d\sp{0}}
\item Finally, we factor \w{\ast\to\hWu{n}} as a cofibration followed by trivial
(Reedy) fibration \w[,]{q\bp{n}:\W{n}~\xepic{\simeq}~\hWu{n}} so
\w{\prn{n}:=r\bp{n}\circ q\bp{n}} is the required Reedy fibration of
\wref[.]{eqtower}
\end{enumerate}
\end{mysubsect}

\begin{remark}\label{rinind}
In step (b), we construct the map \w{\Fpp:C\sp{\ast}\W{n-1}\to\Ds}
by a downward induction on the dimension \w[,]{k\leq n-1} starting with
\w[,]{\Fpk{n-1}:C\sp{n-1}\W{n-1}\to\cD\sp{n-1}=\oW{n}} which exists by
\cite[Lemma 2.19]{BSenH}.

At the $k$-th stage in the induction we have \w{\Fpk{k+1}} and \w{\Fpk{k}}
in the following diagram:
\mydiagram[\label{eqcochainmap}]{
C\sp{k+1}\W{n-1}\ar[rrr]\sp{\Fpk{k+1}} &&& P\Omega\sp{n-k-3}\oW{n} & =~\cD\sp{k+1} \\
&&& \Omega\sp{n-k-2}\oW{n} \ar@{^{(}->}[u] & \\
C\sp{k}\W{n-1} \ar[uu]\sp{\delta\sp{k}} \ar[rrr]\sp{\Fpk{k}} &&&
P\Omega\sp{n-k-2}\oW{n} \ar@{->>}[u] &
=~\cD\sp{k} \ar[uu]\sb{\delta\sp{k}\sb{\cD}} \\
&&& \Omega\sp{n-k-1}\oW{n} \ar@{^{(}->}[u] \ar@/_{3.9pc}/[uu]\sb{0} & \\
C\sp{k-1}\W{n-1}
\ar@{-->}[rrru]\sp{a\sp{k-1}} \ar[uu]\sp{\delta\sp{k-1}}
\ar@{.>}[rrr]_(0.5){\Fpk{k-1}} &&& P\Omega\sp{n-k-1}\oW{n} \ar@{->>}[u] &
=~~\cD\sp{k-1} \ar[uu]\sb{\delta\sp{k-1}\sb{\cD}}
}
\noindent We see that \w{\Fpk{k}} induces a map \w{a\sp{k-1}} as indicated, which
must be nullhomotopic in order for \w{\Fpk{k-1}} to exist. In fact, we have:
\end{remark}

\begin{prop}\label{pvanish}
Let \w{R=\Fp} or $\QQ$ and \w[,]{\Gamma=\HiR{\bY}} and let \w{\Vd\to\Gamma}
be a CW resolution. Assume we have an \wwb{n-1}stage
coaugmented realization \w{\bY\to\W{n-1}} of \w{\Vd} as above, with
\w{\Ds} a Reedy fibrant cochain complex, as well as a cochain map
\w{\Fpp:C\sp{\ast}\W{n-1}\to\Ds} as in
\wref[,]{eqcochainmap} defined in degrees $\geq k$. Then one can modify the choice
of \w{\Fpk{k}} so that \w{a\sp{k-1}} defined as above is nullhomotopic.
\end{prop}

\begin{proof}
This follows from the proof of \cite[Theorem A.11]{BSenH}.
\end{proof}

\begin{corollary}\label{creal}
For \w[,]{\Gamma=\HiR{\bY}} any CW resolution \w{\Vd\to\Gamma} as above
is realizable by a coaugmented cosimplicial space \w{\bY\to \Wu} obtained as a limit
of a tower \wref{eqtower} as above.
\end{corollary}

\begin{mysubsection}{Higher cohomology operations}\label{shco}
We think of \w{a\sp{k-1}} as the value of a \wwb{n-k}th order cohomology operation;
\cite[Theorem A.11(b)]{BSenH} then shows that, given $\bY$, this higher order
operation vanishes, so $F$ exists.

However, given another $R$-good space $\bZ$ with
\w[,]{\HsR{\bZ}\cong\HsR{\bY}} we can try to construct a coaugmentation
\w{\bv:\bZ\to\Wu} inducing a weak equivalence to \w[:]{\Tot\Wu\simeq\hY}
this is possible if and only if $\bZ$ and $\bY$ are $R$-equivalent.
This will be carried out by inductively attempting to produce successive
lifts \w[,]{\bve{k}:\bZ\to\W{k}} starting with the obvious
\w[.]{\bve{0}:\bZ\to\W{0}=\cu{\oW{0}}}

Given \w[,]{\bve{n-1}} consider the composite $\xi$ of
\begin{myeq}\label{eqnegone}
\bZ~\xra{\bve{n-1}}~\Wn{0}{n-1}~\xra{\Fpk{0}}~P\Omega\sp{n-2}\oW{n}~
\xra{p}~\Omega\sp{n-2}\oW{n}~\xra{\iota}~P\Omega\sp{n-3}\oW{n}~,
\end{myeq}
\noindent which represents the component of the iterated coface map
\w{d\sp{1}\circ d\sp{0}\circ \bve{n-1}} from $\bZ$ into \w[.]{P\Omega\sp{n-3}\oW{n}}
Since \w{d\sp{1}\circ d\sp{0}\circ \bve{n-1}=d\sp{2}\circ d\sp{1}\circ \bve{n-1}}
and \w{d\sp{2}=0} into the factor \w[,]{P\Omega\sp{n-3}\oW{n}} we see that $\xi$
is zero. Since $\iota$ is monic, this means that \w{p\circ\Fpk{0}\circ\bve{n-1}}
is already zero, so \w{\Fpk{0}\circ\bve{n-1}} factors through the fiber
\w{\Omega\sp{n-1}\oW{n}} of the path fibration $p$. We denote the resulting map by
\w[.]{a\sp{-1}:\bZ\to\Omega\sp{n-1}\oW{n}} This is the obstruction
to lifting \w{\bve{n-1}} to \w{\bve{n}} (see \cite[Lemma 4.5]{BSenH}).

Since all but the first map in \wref{eqnegone} are $R$-GEMs, from the discussion
in \cite[\S 4]{BSenH} we see that \w{[a\sp{-1}]} can indeed be interpreted as
the value of an appropriate $n$-th order cohomology operation. 

Note that in \wref{eqcochainmap} the various spaces \w{C\sp{j}\W{n-1}} are
not $R$-GEMs. Nevertheless, that realization problem which we solve by
the vanishing of the classes \w[,]{a\sp{k-1}} for the chain map 
\w{\phi:\oV{n}\otimes S\sp{n-1}\to C\sb{\ast}\sk{n-1}\Vd} of \S \ref{srcwr}(b),
is equivalent to that of realizing the $n$-skeletal augmented simplicial object
\w{\sk{n}\Vd\to\Gamma} in \w[,]{s\RAlg} which can be thought of as
an $n$-skeletal augmented cosimplicial object in \w[,]{c\ho\Sa} for which indeed all 
but one object are $R$-GEMs.
\end{mysubsection}

\begin{remark}\label{rhcocomp}
Note that because we are mapping into $R$-GEMs, from the universal property
of the $R$-completion we see that the value of \w{[a\sp{-1}]} depends only on the
$R$-type of $\bZ$ (that is, up to zigzags of maps inducing isomorphisms in
\w[).]{\HiR{-}}  In particular, since the spaces in \wref{eqtower} are all
coaugmented out of $\bY$, all these higher operations indeed vanish for $\bY$ \wh
and thus also for its $R$-completion \w[.]{\hY\simeq\Tot\Wu} Thus $\hY$ too is
coaugmented into \wref[,]{eqtower} and thus into \w[.]{\Wu}

This is a somewhat unusual situation, since \w{\Tot\Wu} is a homotopy limit,
and we would not generally expect a map \w{\Deu\to\Wu} to lift through
the natural map \w{\lim\to\holim} to an actual cone for
\w{\Wu} \wwh that is, to a map \w[.]{\ast\to \Wu}
\end{remark}

%
%
\sect{The homotopy spectral sequence of a cosimplicial space}
\label{chsscs}

For any fibrant pointed cosimplicial space \w[,]{\Wu} Bousfield and Kan construct
a spectral sequence as follows:

\begin{mysubsection}{The $\Tot$ tower}\label{stott}
In the version of \cite[X, \S 6]{BKanH}, this is just the homotopy spectral sequence
of the tower of fibrations:
\begin{myeq}\label{eqtott}
\dotsc\Tot\sb{n+1}\Wu\xra{q\sp{n+1}}\Tot\sb{n}\Wu\xra{q\sp{n}}
\Tot\sb{n-1}\Wu\dotsc\to\Tot\sb{-1}\Wu=\ast,
\end{myeq}
\noindent with (homotopy) limit \w[.]{\Tot\Wu}

Recall that \w{\Tot\Wu:=\map\sb{c\Ss}(\Deu,\,\Wu)} (the simplicial enrichment
of \w[),]{c\Ss} where \w{\Deu} is the cosimplicial space with
\w{\Del{k}} (the standard $k$-simplex) in dimension $k$, and similarly
\w[.]{\Tot\sb{n}\Wu:=\map\sb{c\Ss}(\sk{n}\Deu,\,\Wu)}
One should think of a map \w{\bZ\to\Tot\Wu} as an $\infty$-homotopy commutative
diagram mapping $\bZ$ into the $\Delta$-indexed diagram \w[.]{\Wu}

We shall use \w{\Dlt{k}} as an alternative notation for the standard $k$-simplex in
$\Ss$, when we think of as representing $k$-simplices in simplicial sets.
In particular, a $k$-simplex in \w{\Tot\sb{n}\Wu} is a sequence of maps
\w{f\sp{m}:\sk{n}\Del{m}\times\Dlt{k}\to\bW\sp{m}} \wb{m=0,1,\dotsc} such that
\begin{myeq}\label{eqelttot}
f\sp{j}\circ(\sk{n}\!\Del{}(\phi)\times\Id)~=~\bW(\phi)\circ f\sp{m}~:~
\sk{n}\!\Del{m}\times\Dlt{k}\to\bW\sp{j}
\end{myeq}
\noindent for every morphism \w{\phi:[m]\to[j]} in $\Delta$. Therefore,
for each \w{k\geq 1} we have
\w[,]{f\sp{n-k}=s\sb{\bW}\sp{I}\circ f\sp{n}\circ d\sb{\bDel}\sp{I}} where
\w{d\sb{\bDel}\sp{I}=d\sb{\bDel}\sp{i\sb{k}}\circ\dotsc d\sb{\bDel}\sp{i\sb{1}}}
is an iterated coface map of $\bDel$, and \w{s\sp{I}\sb{\bW}} is the corresponding
iterated codegeneracy map of \w{\Wu} (since \w[).]{s\sp{I}\circ d\sp{I}=\Id}
Moreover, since \w{\sk{n}\Del{N}=\colim\sb{i\leq n}\Del{i}} for \w[,]{N>n}
the map \w{f\sp{N}} is determined by the (compatible) maps \w{f\sp{i}} for
\w[.]{i\leq n} Thus
\w{(\Tot\sb{n}\Wu)\sb{k}\subseteq\Hom(\Del{n}\times\Dlt{k},\,\bW\sp{n})} and in fact
\begin{myeq}\label{eqtotn}
\Tot\sb{n}\Wu\subseteq\map\sb{\Ss}(\Del{n},\,\bW\sp{n})~,
\end{myeq}
\noindent where the subspace is the limit given by \wref[.]{eqelttot}

Note that because \w{\sk{n}\Del{n+1}=\partial\Del{n+1}} and each of the coface
maps \w{d\sp{i}:\bW\sp{n}\to\bW\sp{n+1}} has a retraction, the compatibility
conditions mean that also
\begin{myeq}\label{eqtotnp}
\Tot\sb{n}\Wu\subseteq\map\sb{\Ss}(\partial\Del{n+1},\,\bW\sp{n+1})~.
\end{myeq}

As in \cite[X, 6.3]{BKanH}, the map \w{q\sp{n}} of \wref{eqtott}
fits into a fibration sequence
\begin{myeq}\label{eqtotfs}
\Omega\sp{n}N\sp{n}\Wu~\stackrel{\iota\sp{n}}{\hra}~\Tot\sb{n}\Wu~\xepic{q\sp{n}}~
\Tot\sb{n-1}\Wu~,
\end{myeq}
\noindent where
\begin{myeq}\label{eqnwu}
N\sp{n}\Wu~:=~\bW\sp{n}\cap\Ker(s\sp{0})\cap\dotsc\cap\Ker(s\sp{n-1})
\end{myeq}
\noindent (and thus \w[).]{\Omega\sp{n}N\sp{n}\Wu=N\sp{n}\Omega\sp{n}\Wu}

Furthermore, by \wref{eqtotn} and \wref[,]{eqtotnp} the sequence \wref{eqtotfs}
is just the restriction of the fibration sequence:
\begin{myeq}\label{eqbigfs}
\map\sb{\Sa}(\Del{n}/\partial\Del{n},\,\bW\sp{n})~\xra{p\sb{n}\sp{\ast}}~
\map\sb{\Ss}(\Del{n},\bW\sp{n})~\xra{\iota\sb{n}\sp{\ast}}~
\map\sb{\Ss}(\partial\Del{n},\bW\sp{n})
\end{myeq}
\noindent induced by the cofibration sequence
$$
\partial\Del{n}~\xra{\iota\sb{n}}~\Del{n}~\xepic{p\sb{n}}~\Del{n}/\partial\Del{n}~.
$$
\end{mysubsection}

\begin{remark}\label{rhsstow}
Since \w{\Wu} is pointed, \w{\map\sb{\Ss}(\Del{n},\,\bW\sp{n})} has a chosen
basepoint, and an element in \w{\pi\sb{k}\Tot\sb{n}\Wu} is represented by a suitable
pointed map \w[,]{f:\bS{k}\to\map\sb{\Ss}(\Del{n},\,\bW\sp{n})} or by its (pointed)
adjoint \w{\hat{f}:\hfsm{\Del{n}}{\bS{k}}\to\bW\sp{n}} (cf.\ \S \ref{snac}). Note
that the maps into \w{\bW\sp{j}} \wb{0\leq j<n} are encoded by maps into the
appropriate codegeneracies in \w[.]{\bW\sp{n}}

Thus a class
\w{\alpha\in\pi\sb{k}\Omega\sp{n}N\sp{n}\Wu\subseteq\pi\sb{k}\Tot\sb{n}\Wu}
is represented by \w{a:\partial\Del{n+1}\times\Dlt{k}\to\bW\sp{n+1}} which
vanishes on \w[.]{\partial\Del{n+1}\times\partial\Dlt{k}}
Such an $\alpha$ represents an element \w{\gamma\in\pi\sb{k}\Tot\Wu} if and only
if \w{j\sp{n}\sb{\#}(\alpha)} lifts to all levels of \wref[,]{eqtott} where
\w{j\sp{n}\sb{\#}:\pi\sb{k}\Omega\sp{n}N\sp{n}\Wu\to\pi\sb{k}\Tot\sb{n}\Wu}
is induced by the inclusion. The successive obstructions to lifting
\w{j\sp{n}(\alpha)} represent the differentials in the spectral sequence.
\end{remark}

\begin{mysubsection}{The spectral sequence}\label{sss}
The \ww{E\sb{1}}-exact couple of the homotopy spectral sequence for the $\Tot$ tower
\wref{eqtott} may be presented as in Figure \ref{fig1}:

%
%
\begin{figure}[htbp]
\begin{center}
\xymatrix@R=25pt@C=11pt{
\pi\sb{k+1}\Tot\sb{n}\Wu \ar[d]\sp{q\sp{n}} \ar[rr]\sp{\delta\sp{n}} &&
\pi\sb{k}\Omega\sp{n+1}N\sp{n+1}\Wu \ar[rr]\sp{j\sp{n+1}} &&
\pi\sb{k}\Tot\sb{n+1}\Wu \ar[d]\sp{q\sp{n+1}} \ar[rr]\sp(0.45){\delta\sp{n+1}} &&
\pi\sb{k-1}\Omega\sp{n+2}N\sp{n+2}\Wu\\
\pi\sb{k+1}\Tot\sb{n-1}\Wu \ar[rr]\sp{\delta\sp{n-1}} &&
\pi\sb{k}\Omega\sp{n}N\sp{n}\Wu \ar[rr]^{j\sp{n}} &&
\pi\sb{k}\Tot\sb{n}\Wu \ar[rr]\sp(0.45){\delta\sp{n}} && \pi\sb{k-1}\Omega\sp{n+1}N\sp{n+1}\Wu
}
\end{center}
\caption[fig1]{Exact couple for \w{\Tot} tower}
\label{fig1}
\end{figure}
\noindent with \w{E\sp{n,n+k}\sb{1}:=\pi\sb{k}\Omega\sp{n}N\sp{n}\Wu} and
\ww{d\sb{1}}-differential given by
\begin{myeq}\label{eqdone}
d\sb{1}\sp{n,n+k}~=~\delta\sp{n}\circ j\sp{n}~=~
\sum\sb{t=0}\sp{n-1}\,(-1)\sp{t}d\sp{t}\sb{\#}:
\pi\sb{k+n}N\sp{n}\Wu\to\pi\sb{k+n}N\sp{n+1}\Wu
\end{myeq}
\noindent
by \cite[X, 6.3]{BKanH} again.
\end{mysubsection}

%
%
\sect{The unstable Adams spectral sequence}
\label{cuass}

From now on we restrict attention to the case \w{R=\Fp} (although most results are
valid also for \w[).]{R=\QQ}
When \w{\Wu} is an \ww{\Fp}-resolution of a $p$-good space $\bY$, the homotopy
spectral sequence of Section \ref{chsscs} is the unstable Adams spectral sequence
of \cite{BKanH}, converging to \w{\pi\sb{\ast}\hY} (where
\w{\hY:=R\sb{\infty}\bY} is the $p$-completion, equipped with the natural
\ww{\HsR{-}}-equivalence \w[).]{\eta:\bY\to\hY} In this case we can say a little more
about the structure of the spectral sequence:

\begin{mysubsection}{Using the CW structure}\label{sucws}
From now on, we assume that \w{\Wu} has been constructed as in \S \ref{srcwr}
to realize a given CW resolution \w{\Vd} of \w[.]{\Gamma=\HsR{\bY}}
From \wref{eqdopb} and \wref{eqsmatch} we see that the
maps \w{\prn{n+1}:\W{n+1}\to\W{n}} in \wref{eqtower} induce weak equivalences in
\w{\Tot\sb{k}} for all \w[.]{0\leq k\leq n}
Since \w{\W{n}} is $n$-coskeletal, we have a tower of fibrations:
\begin{myeq}\label{eqmodtott}
\dotsc\Tot\sb{n+1}\W{n+1}\xra{(\prn{n+1})\sb{\ast}}
\Tot\sb{n}\W{n}\xra{(\prn{n})\sb{\ast}}\dotsc
\Tot\sb{1}\W{1}\xra{(\prn{1})\sb{\ast}}\Tot\sb{0}\W{0}~,
\end{myeq}
\noindent obtained by combining \wref{eqtott} and \wref[.]{eqtower}
\end{mysubsection}

In order to better understand the tower \wref[,]{eqmodtott} we recall a
(somewhat simplified) version of a construction introduced in \cite[\S 5.10]{BSenH}:

\begin{defn}\label{dfoldp}
For each \w{n\geq 1} and \w[,]{1\leq k\leq n+1} the $n$-th \emph{folding polytope}
\w{\PP{n}{k}} is obtained from a union of $k$ disjoint $n$-simplices
\w{\Dels{n}{n-k+1},\dotsc,\Dels{n}{n}} by identifying the $j$-th facets
of \w{\Dels{n}{n-j}} and \w{\Dels{n}{n-j-1}} for each
\w[.]{0\leq j\leq n} See Figure \ref{fig2} below for an example.
\end{defn}

\begin{remark}\label{rfoldp}
By induction on \w{1\leq k\leq n} we readily see that \w{\PP{n}{k}} is
PL-equivalent to an $n$-ball, so its boundary \w{\partial\PP{n}{k}} is PL-equivalent
to an \wwb{n-1}-sphere.
\end{remark}

\begin{lemma}\label{ldtot}
For \w[,]{\Wu} \w[,]{\Du{n}} and \w{\sDu{n}} as in \S \ref{srcwr},
\w{\Tot\Du{n}\simeq\Omega\sp{n-1}\oW{n}} and
\w[.]{\Tot\sDu{n}\simeq\Omega\sp{n}\oW{n}}
\end{lemma}

\begin{proof}
Since \w{\Du{n}} is \wwb{n-1}coskeletal, \w[.]{\Tot\Du{n}=\Tot\sb{n-1}\Du{n}}
Moreover, by \S \ref{stott}, and \S \ref{srcwr}(d), for any \w[,]{\bZ\in\Sa} a
pointed map \w{g:\bZ\to\Tot\sb{n-1}\Du{n}} is completely determined by a sequence
of maps \w{g\sp{j}:\hfsm{\Del{j}}{\bZ}\to P\Omega\sp{n-j-2}\oW{n}}
\wb[,]{0\leq j\leq n-1} making
\mytdiag[\label{eqmaptofib}]{
\hfsm{\Del{j}}{\bZ} \ar[rrr]\sp{g\sp{j}} &&& P\Omega\sp{n-j-2}\oW{n} \\
\hfsm{\Del{j-1}}{\bZ}
\ar@<1.5ex>[u]\sp(0.5){\delta\sp{0}}\sb(0.5){\dotsc}
\ar@<-1.5ex>[u]\sb(0.5){\delta\sp{j}} \ar[rrr]\sp{g\sp{j-1}} &&&
P\Omega\sp{n-j-1}\oW{n} \ar@<2ex>[u]\sp{\iota\circ p=d\sp{0}}\sb{\dotsc}
 \ar@<-2ex>[u]\sb{0=d\sp{i}\ (i\geq 1)}
}

\noindent commute for each \w{0<j\leq n-1} (where the coface maps \w{\delta\sp{i}}
on the left are induced by those of \w[).]{\Deu} See Remark \ref{rhsstow} and
\wref[.]{eqdk}

Note that the cone functor in \w{\Sa} is left adjoint to $P$, so if we include
each \w{P\Omega\sp{i}\oW{n}} into \w[,]{P\sp{i+1}\oW{n}} and identify
\w{C\Del{j}} with \w[,]{\Del{j+1}} we see that for each \w{0\leq j\leq n-1}
the adjoint of \w{g\sp{j}} is a map \w[.]{\tg{j}:\hfsm{\Del{n-1}}{\bZ}\to\oW{n}}
We arrange the adjunction between \w{g\sp{j}} and \w{\tg{j}} in such a way that the
coface maps \w{\delta\sp{0},\dotsc,\delta\sp{n-j-3}} of \w{\Del{n-1}}
correspond to the loop directions of \w{P\Omega\sp{n-j-2}\oW{n}}
(counted outwards from \w[),]{\oW{n}} and \w{\delta\sp{n-j-2}} corresponds to the
path direction. Finally, as long as \w[,]{j>0} the remaining \w{j+1} coface maps into
\w{\Del{n-1}} are the original coface maps of \w[,]{\Del{j}} re-indexed by
\w[.]{n-j-1}

The fact that \wref{eqmaptofib} commutes implies that these adjoints
satisfy the relations
\begin{myeq}\label{eqcosimprels}
\tg{j}\circ\delta\sp{i}~=~\begin{cases}
\widetilde{\iota p g\sp{j}} & \text{for}\hsm i=n-j-2\hsm \text{and}\ j<n-1\\
\widetilde{\iota p g\sp{j-1}} & \text{for}\hsm i=n-j-1\hsm \text{and}\ j>0\\
0  & \text{otherwise.}
\end{cases}
\end{myeq}

By Definition \ref{dfoldp}, the maps \w{\tg{j}} thus induce a
single map \w{\tg{}:\hfsm{\PP{n-1}{n}}{\bZ}\to\oW{n}} Moreover, \wref{eqcosimprels}
also implies that \w[,]{\tg{}\rest{\hfsm{\partial\PP{n-1}{n}}{\bZ}}=0} so $\tg{}$
factors uniquely through a map
\w[.]{\sms{(\PP{n-1}{n}/\partial\PP{n-1}{n})}{\bZ}\to\oW{n}}
By Remark \ref{rfoldp}, \w{\PP{n-1}{n}/\partial\PP{n-1}{n}} is a PL \wwb{n-1}sphere,
so setting \w{\bZ=\bS{i}} we see that
\w{\Tot\Du{n-1}} is weakly equivalent to
\w[.]{\Omega\sp{n-1}\oW{n}}

Similarly, \w{\sDu{n}} is $n$-coskeletal, so \w[,]{\Tot\sDu{n}=\Tot\sb{n}\sDu{n}}
and a map of simplicial sets \w{g:\bZ\to\Tot\sb{n}\sDu{n}} is determined
(via the codegeneracies) by maps
\w{g\sp{j}:\hfsm{\Del{j}}{\bZ}\to P\Omega\sp{n-j-1}\oW{n}}
making the following diagram commute:
\mytdiag[\label{eqmaptosfib}]{
\hfsm{\Del{j}}{\bZ} \ar[rrr]\sp{g\sp{j}} &&& P\Omega\sp{n-j-1}\oW{n} \\
\hfsm{\Del{j-1}}{\bZ}
\ar@<1.5ex>[u]\sp(0.5){\delta\sp{0}}\sb(0.5){\dotsc}
\ar@<-1.5ex>[u]\sb(0.5){\delta\sp{j}} \ar[rrr]\sp{g\sp{j-1}} &&&
P\Omega\sp{n-j}\oW{n}~. \ar@<2ex>[u]\sp{\iota\circ p=d\sp{1}}\sb{\dotsc}
 \ar@<-2ex>[u]\sb{0=d\sp{i}\ (i\neq 1)}
}
\noindent See \S \ref{srcwr}(d).

Taking adjoints \w{\tg{j}:\hfsm{\Del{n}}{\bZ}\to\oW{n}}
\wb{0\leq j\leq n} as above, \wref{eqcosimprels} is replaced by:
\begin{myeq}\label{eqcosimprelas}
\tg{j}\circ\delta\sp{i}~=~\begin{cases}
\widetilde{\iota p g\sp{j}} & \text{for}\hsm i=n-j-1\hsm \text{and}\ j<n\\
\widetilde{\iota p g\sp{j-1}} & \text{for}\hsm i=n-j+1\hsm \text{and}\ j>0\\
0  & \text{otherwise.}
\end{cases}
\end{myeq}
\noindent and as before we deduce that
\begin{myeq}\label{eqsusd}
\Tot\sb{n}\sDu{n}~=~\Tot\sDu{n}~\simeq~\Omega\sp{n}\oW{n}~.
\end{myeq}
\end{proof}

\begin{remark}\label{rnewfold}
We see from \wref{eqcosimprelas} that the folding polytopes used to show that
\wref{eqsusd} holds are different from those defined in \S \ref{dfoldp},
since we need to identify the $j$-th facet of \w{\Dels{n}{n-j+1}} with
the \wwb{j-1}facet of  \w{\Dels{n}{n-j}} for each \w[.]{0<j\leq k}
We denote this variant by \w[,]{\hP{n}{k}} which we called a \emph{modified} folding
polytope. See Figure \ref{fig3} below for an example.
\end{remark}

\begin{prop}\label{pqfib}
For \w{\Wu} as above, the sequence of maps of \S \ref{srcwr}(d) induce a
quasifibration sequence
\begin{myeq}\label{eqqfib}
\Tot\sb{n}\sDu{n}~\xra{\Tot i\bp{n}}~\Tot\sb{n}\hWu{n}~\xra{\Tot r\bp{n}}~
\Tot\sb{n-1}\W{n-1}~\xra{\Tot \bF\bp{n}}~\Tot\sb{n-1}\Du{n}~.
\end{myeq}
\end{prop}

\begin{proof}
As noted in \S \ref{srcwr}, \w{\sDu{n}\xra{i\bp{n}}\hWu{n}\xra{r\bp{n}}\W{n-1}} is
a Reedy fibration sequence of Reedy fibrant cosimplicial sets, and \w{\W{n-1}} is
\wwb{n-1}coskeletal, so applying \w{\Tot} yields exactness at the left three terms of
\wref[.]{eqqfib}

For the right three terms, note that for any pointed space $\bZ$ a map
\w{g:\bZ\to\Tot\sb{n-1}\W{n-1}} is described by
\w{g\sp{k}:\hfsm{\Del{k}}{\bZ}\to\Wn{k}{n-1}} for \w[,]{0\leq k\leq n-1}
as in the proof of Lemma \ref{ldtot}. Moreover, the reduced cone on the half-smash:
\w{\bar{C}(\hfsm{\bX}{\bY})} (where \w{\bY\in\Sa} is pointed, but \w{\bX\in\Ss}
is not) is isomorphic to \w{\sms{C\bX}{\bY}} (the smash product with the unreduced
cone on $\bX$ \wh cf.\ \S \ref{snac}).

So a nullhomotopy
\w{H:\bF\bp{n}\circ g\sim 0} is determined by a sequence of maps
\w{H\sp{k}:\hfsm{\Del{k}}{\bZ}\to P\Omega\sp{n-k-1}\oW{n}} for \w[,]{1\leq k\leq n}
and the following diagram must commute for each $k$, as in \wref[:]{eqmaptofib}
\mytdiag[\label{eqnulhtpy}]{
\sms{C\Del{k}}{\bZ} \ar@/^{1.9pc}/[rrrrrd]\sp{H\sp{k}} &&&&&\\
& \hfsm{\Del{k}}{\bZ} \ar@{_{(}->}[ul]\sb{\delta\sp{0}} \ar[rr]\sp{g\sp{k}} &&
\Wn{k}{n-1} \ar[rr]\sp(0.45){\Fk{k}} && P\Omega\sp{n-k-2}\oW{n} \\
\sms{C\Del{k-1}}{\bZ}
\ar@<2.5ex>[uu]\sp(0.55){Cd\sp{0}}\sp(0.45){=\delta\sp{1}}\sb(0.5){\dotsc}
 \ar@<-2.5ex>[uu]\sb(0.55){Cd\sp{k}}\sb(0.45){=\delta\sp{k+1}}
\ar@/^{3.3pc}/[rrrrrd]\sp{H\sp{k-1}} &&&&&\\
& \hfsm{\Del{k-1}}{\bZ} \ar@{_{(}->}[ul]\sb{\delta\sp{0}}
\ar@<2.5ex>[uu]\sp(0.5){d\sp{0}}\sb(0.5){\dotsc} \ar@<-2.5ex>[uu]\sb(0.5){d\sp{k}}
\ar[rr]\sp(0,55){g\sp{k-1}} &&
\Wn{k-1}{n-1} \ar[rr]\sp(0.45){\Fk{k-1}} && P\Omega\sp{n-k-1}\oW{n}
 \ar@<3ex>[uu]\sp(0.6){\iota\sb{n-k-1}\circ p}\sb(0.6){=d\sp{0}}
 \ar@<-4ex>[uu]\sb(0.4){(j>0)}\sb(0.6){=d\sp{j}}\sp(0.6){0}
}
\noindent Here we think of \w{\Del{k}} as the (unreduced) cone \w[,]{C\Del{k-1}} with
\w{\delta\sp{0}:\Del{k-1}\to\Del{k}} the inclusion of the base, and
\w{\delta\sp{j}} the cone on \w{d\sp{j-1}:\Del{k-2}\to\Del{k-1}} for
\w[.]{1\leq j\leq k}
We write \w{\Fk{k}=\Fk{k}\bp{n-1}:\Wn{k}{n-1}\to P\Omega\sp{n-k-2}\oW{n}}
for the composite
\begin{myeq}\label{eqfk}
\Wn{k}{n-1}~\xra{v\sp{k}}~C\sp{k}\W{n-1}~\xra{\Fpk{k}}~P\Omega\sp{n-k-2}\oW{n}~.
\end{myeq}
\noindent See \S \ref{dmoorc} and \S \ref{srcwr}(c).

The maps \w{H\sp{k}} must satisfy:
\begin{myeq}\label{eqnullhtpy}
H\sp{k}\circ\delta\sp{0}=\Fk{k-1}\circ g\sp{k-1},\
H\sp{k}\circ\delta\sp{1}=\iota\sb{n-k-1}\circ p\circ H\sp{k-1},\ \text{and}\
H\sp{k}\circ\delta\sp{j}=0\ \text{for}\ j\geq 2,
\end{myeq}
\noindent where \w{\iota\sb{r}:\Omega\sp{r}\oW{n}\hra P\Omega\sp{r}\oW{n}} is the
inclusion and $p$ is the path fibration. Moreover,
\begin{myeq}\label{eqnullhzero}
H\sp{1}\circ\delta\sp{1}~=~0~,
\end{myeq}
\noindent since $H$ is a nullhomotopy.

On the other hand, a lift of $g$ to \w[,]{h:\bZ\to\Tot\sb{n}\hWu{n}}
is given by a sequence of maps
\w{h\sp{k}:\hfsm{\Del{k}}{\bZ}\to P\Omega\sp{n-k-1}\oW{n}} for \w[,]{0\leq k\leq n}
with
\begin{myeq}\label{eqliftmap}
h\sp{k}\circ\delta\sp{0}~=~\Fk{k-1}\circ g\sp{k-1}~,\hsm
h\sp{k}\circ\delta\sp{1}~=~\iota\sb{n-k-1}\circ  p\circ h\sp{k-1}~,\hsm\text{and}\
h\sp{k}\circ\delta\sp{j}~=~0\hsm \text{for} \ j\geq 2~.
\end{myeq}

Thus, given $H$, we may set \w{h\sp{k}:=H\sp{k}} for \w[.]{1\leq k\leq n}
By \wref[,]{eqnullhzero} we then have
\w{\iota\sb{n-1}\circ p\circ h\sp{0}=0} so \w{h\sp{0}} must factor uniquely as
\w{\bZ\xra{\varphi}\Omega\sp{n}\oW{n}\xra{\iota\sb{n}}P\Omega\sp{n}\oW{n}}
(since \w{\iota\sb{n-1}} is monic).

From the description of \w{g:\bZ\to\Tot\sDu{n}} in the proof of Lemma \ref{ldtot},
we see that for any \w[,]{0\leq m\leq n} any class in \w{[\bZ,\,\Tot\sDu{n}]} may be
represented by a collection of maps
\w{(\widetilde{g}\sb{k}:\hfsm{\Del{n}}{\bZ}\to\oW{n})\sb{k=0}\sp{n}} with
\w{\widetilde{g}\sb{k}=0} for \w{k\neq m} and
\w[.]{\widetilde{g}\sb{m}\rest{\hfsm{\partial\Del{n}}{\bZ}}=0}
For \w[,]{m=0} this shows that the choices for the lift $h$,
given $H$, are uniquely determined by the image under \w{(i\bp{n})\sb{\ast}} of
\w[.]{[g]\in[\bZ,\,\Tot\sb{n}\sDu{n}]}
This completes the proof by showing the exactness at \w{\Tot\sb{n-1}\hWu{n-1}}
in applying \w{[\bZ,-]} to \wref[.]{eqqfib}
\end{proof}

\begin{thm}\label{tuass}
For \w{\Wu} constructed as in \S \ref{srcwr}, the spectral sequence associated
to the tower of fibrations \wref{eqmodtott} agrees from the \ww{E\sb{2}}-term on
with the unstable Adams spectral sequence of \cite[\S 4]{BKanS}.
\end{thm}

\begin{proof}
Because each of the cosimplicial spaces \w[,]{\Wu} \w[,]{\W{n}} and \w{\hWu{n}}
is Reedy fibrant, and the maps \w[,]{\prn{n}} \w[,]{q\bp{n}} and \w{r\bp{n}}
of \S \ref{srcwr} are Reedy fibrations, we have trivial fibrations
\begin{myeq}\label{eqnorch}
N\sp{n}\Wu~\xepic{\simeq}~N\sp{n}\W{n}~\xepic{\simeq}~N\sp{n}\hWu{n}~
\xepic{\simeq}~\oW{n}
\end{myeq}
\noindent for each \w[,]{n\geq 0} since from \wref{eqcindstep} we see that
\w[.]{N\sp{n}\hWu{n}=\prod\sb{r\geq n}\ P\Omega\sp{r-n-1}\oW{r}}

Moreover, by \cite[X, 6.3(ii)]{BKanH} we have
$$
\pi\sb{\ast}\oW{n}~\cong~\pi\sb{\ast}N\sp{n}\Wu~\cong~
N\sp{n}\pi\sb{\ast}\Wu~\cong~C\sp{n}\pi\sb{\ast}\Wu~,
$$
\noindent using the dual of \cite[Lemma 2.11]{BJTurnHA} for the graded
cosimplicial abelian group \w[.]{\pi\sb{\ast}\Wu}

Finally, from the fact that \w{\HiR{\Wu}\cong\Vd} (a free \TRal resolution of
\w[),]{\Gamma=\HiR{\bY}} and that, as in \cite[X, \S 7]{BKanH}, the
\ww{d\sb{1}}-differential of \wref{eqdone} reduces to \w[,]{d\sp{0}\sb{\#}}
coming from the CW attaching map of \w{\Vd} (\S \ref{dcwres}), we conclude
from \cite[X, \S 6.4]{BKanH} that we have a natural isomorphism between
our \ww{E\sb{2}}-term and \w[,]{\pi\sp{\ast}\pi\sb{\ast}\Wu} which
is isomorphic in turn to that of the unstable Adams spectral sequence by
\cite[\S 10.2]{BKanS}.
\end{proof}

%
%
\sect{Differentials in the unstable Adams spectral sequence}
\label{cduass}

In order to describe the differentials in the homotopy spectral sequence for the
tower \wref[,]{eqmodtott} we associate to every \wwb{n,k} slot in the spectral
sequence a sequence of $r$-th order cohomology operations
\w{\llrr{-}{r}:\pi\sb{k+n}\oW{n}\to\pi\sb{k+n+r-1}\oW{n+r}} for \w[,]{r\geq 1}
as described in \S \ref{sshoo}.

\begin{mysubsection}{Differentials and higher cohomology operations}\label{sdhco}
These operations are constructed inductively by a sequence of choices, starting
with (but independent of) a representative of \w[.]{\gamma\in\pi\sb{k+n}\oW{n}}
called the \emph{data} for  \w[.]{\llrr{\gamma}{r}} In particular, for each
\w[,]{r\geq 1}
\w{\llrr{\gamma}{r+1}} is defined only if \w{\llrr{\gamma}{r}} vanishes, and the data
for the former includes a choice of a nullhomotopy \w{\Hv{n+r}} for the latter value.

The choice of \w{\Hv{n+r}} defines a certain $\infty$-homotopy commutative
diagram (in the form of a map \w[),]{\hGv{n+r}:\hfsm{\Deu}{\bS{k}}\to\hWu{n+r}}
which we then make cofibrant (as a map
\w[),]{\Gv{n+r}:\hfsm{\Deu}{\bS{k}}\to\W{n+r}} yielding an appropriate value for
\w[.]{\llrr{\gamma}{r+1}}
\end{mysubsection}

\begin{mysubsection}{The inductive construction}\label{sindcon}
We want to associate to every \wwb{n,k} slot in the spectral sequence for
\w{\Wu} a sequence of $r$-th order cohomology operations
\w{\llrr{-}{r}:\pi\sb{k+n}\oW{n}\to\pi\sb{k+n+r-1}\oW{n+r}} for \w[,]{r\geq 1}
as described in \S \ref{sshoo}.

We start by representing
\w{\gamma\in E\sb{1}\sp{n,k+n}=\pi\sb{k+n}\oW{n}=\pi\sb{k}\Omega\sp{n}\oW{n}}
by a map \w[,]{h:\bS{k}\to\Omega\sp{n}N\sp{n}\oW{n}} using \wref[.]{eqnorch}
Postcomposing $h$ with the inclusion
\w{\iota\sp{n}:\Omega\sp{n}N\sp{n}\oW{n}\hra\Tot\sb{n}\W{n}} from \wref{eqtotfs}
and the identification \w{\Tot\sb{n}\W{n}\cong\Tot\W{n}} of
Proposition \ref{pqfib}, we obtain \w{h':\bS{k}\to\Tot\W{n}} and so by adjunction
$$
\Gv{n}:\hfsm{\Deu}{\bS{k}}\to\W{n}~.
$$
\noindent By \wref[,]{eqtotfs} we may assume that
\w{\Gn{i}{n}:\hfsm{\Del{i}}{\bS{k}}\to\Wn{i}{n}} is zero for \w[.]{i<n}

At the $r$-th stage, let \w[,]{N:=n+r-1} and assume by induction that we have lifted
$\gamma$ (that is, \w[)]{\Gv{n}} along \wref{eqtower} to
\w[,]{\Gv{N}:\hfsm{\Deu}{\bS{k}}\to\W{N}} again with
\w{\Gn{j}{N}:\hfsm{\Del{j}}{\bS{k}}\to\Wn{j}{N}} equal to zero for \w[.]{j<n}
By Proposition \ref{pqfib}, \w{\Gv{N}} can be lifted to \w{\hGv{N+1}}
(and thus to \w[),]{\Gv{N+1}} up to homotopy, if and only if
\w[.]{\Fv{N}\circ\Gv{N}\sim 0} We wish to identify the obstruction to the existence
of a nullhomotopy
$$
\Hv{N}:C\Deu\wedge\bS{k}\to\Du{N}
$$
\noindent as an $r$-th order cohomology operation.

Note that \w{\Hv{N}} is completely determined by its projection on the
non-codegenerate factors of \w[,]{\Du{N}} namely,
\w{\Hn{j}{N}:\sms{C\Del{j}}{\bS{k}}\to P\Omega\sp{N-j-1}\oW{N+1}}
(cf.\ \wref[,]{eqdk} and compare \wref[):]{eqnulhtpy}
\mytdiag[\label{eqnullhmtpy}]{
\sms{C\Del{j}}{\bS{k}} \ar@/^{1.9pc}/[rrrrrd]\sp{\Hn{j}{N}} &&&&&\\
& \hfsm{\Del{j}}{\bS{k}} \ar@{_{(}->}[ul]\sb{\delta\sp{0}}
\ar[rr]\sp{\Gv{N}\sp{j}} &&
\Wn{j}{N} \ar[rr]\sp(0.4){\Fk{j}} && P\Omega\sp{N-j-1}\oW{N+1} \\
\sms{C\Del{j-1}}{\bS{k}}  \ar@<2.5ex>[uu]\sp{Cd\sp{0}=\delta\sp{1}}\sb{\dotsc}
 \ar@<-1.0ex>[uu]\sb{\delta\sp{j+1}=Cd\sp{j}}
\ar@/^{2.9pc}/[rrrrrd]\sp{\Hn{j-1}{N}} &&&&&\\
& \hfsm{\Del{j-1}}{\bS{k}} \ar@{_{(}->}[ul]\sb{\delta\sp{0}}
\ar@<2.5ex>[uu]\sp(0.4){d\sp{0}}\sb(0.4){\dotsc} \ar@<-2.5ex>[uu]\sb(0.4){d\sp{j}}
\ar[rr]\sp{\Gv{N}\sp{j-1}} &&
\Wn{j-1}{N} \ar[rr]\sp{\Fk{j-1}}
\ar@<2.5ex>[uu]\sp(0.3){d\sp{0}}\sb(0.3){\dotsc} \ar@<-2.5ex>[uu]\sb(0.3){d\sp{j}}
&& P\Omega\sp{N-j}\oW{N+1}
 \ar@<3ex>[uu]\sp(0.6){\iota\sb{n-j-1}\circ p}\sb(0.6){=d\sp{0}}
 \ar@<-4ex>[uu]\sb(0.6){=d\sp{i}}\sb(0.45){(i\geq 1)}\sp(0.6){0}
}
\noindent with \w{\Fk{j}=\Fk{j}\bp{N}} given by \wref[.]{eqfk}

As in the proof of Lemma \ref{ldtot} (see also \cite[\S 5]{BSenH}), by adjointing
each of the path or loop directions of \w{P\Omega\sp{N-j-1}\oW{N+1}} to a cone
direction on \w[,]{\Del{j}} we can replace \w{\Hn{j}{N}} by
\w{\tHn{j}{N}:\hfsm{\Del{N+1}}{\bS{k}}\to\oW{N+1}} for \w{n\leq j\leq N}
(since \w{\Gn{j}{N}=0} for \w[,]{j<n} we may assume the same for \w[,]{\Hn{j}{N}}
and thus \w[).]{\tHn{j}{N}}

We retain the conventions of the proof of Proposition \ref{pqfib}: thus
the first facet of \w{\Del{N+1}} is the base of the cone \w[,]{C\Del{j}}
the facets \w{1,\dotsc,N-j-1} of \w{\Del{N}} correspond to the loop directions of
\w{P\Omega\sp{N-j-1}\oW{N+1}}
(counted outwards from \w[),]{\oW{N+1}} with the \wwb{N-j}th facet corresponding
to the path direction. As long as \w[,]{j>0} the next \w{j+1} faces of
\w{\Del{N+1}} are the original faces of \w[,]{\Del{j}} re-indexed by \w[.]{N-j}

The fact that \w{\Hv{N}} is a map of cosimplicial spaces then translates into the
following conditions:
\begin{myeq}\label{eqcosimap}
\tHn{j}{N}\circ\delta\sp{i}~=~\begin{cases}
\widetilde{\Fk{j} G\sp{j}} & \text{for}\hsm i=0\\
\widetilde{\iota p H\sp{j}} & \text{for}\hsm i=N-j\hsm \text{and}\hsm j<N\\
\widetilde{\iota p H\sp{j-1}} & \text{for}\hsm i=N-j+1\hsm \text{and}\hsm j>0\\
0  & \text{otherwise.}
\end{cases}
\end{myeq}
\noindent for each \w[.]{n\leq j\leq N}

Thus we see that \w{\Hv{N}} defines a map
\w{\tHn{}{N}:\hfsm{\PP{N+1}{r}}{\bS{k}}\to\oW{N+1}} (cf.\ \S \ref{dfoldp}), since the
maps \w{\tHn{j}{N}} on \w{\sms{\Dels{N+1}{j}}{\bS{k}}} agree on the identified
facets.

Note that from \wref{eqcosimap} we see that the map \w[,]{\tHn{}{N}}
when restricted to \w[,]{\hfsm{\partial\PP{N+1}{r}}{\bS{k}}} depends only on the given
map \w{\Fv{N}} and the chosen lift \w[,]{\Gv{N}} so we may denote it by
$$
\Phi'\sb{(F,G)}:\hfsm{\partial\PP{N+1}{r}}{\bS{k}}\to\oW{N+1}~.
$$
\noindent Moreover, \w{\Phi'\sb{(F,G)}} is zero on \w{\{v\}\times\bS{k}} for each
of the cone vertices $v$ of \w{\Dels{N+1}{j}} in \w{\PP{N+1}{r}} (because our maps
were defined on the smash product with the cone). Thus \w{\Phi'\sb{(F,G)}} induces
a map
$$
\Phi\sb{(F,G)}:\sms{\partial\PP{N+1}{r}}{\bS{k}}\to\oW{N+1}~.
$$
\noindent Its domain is a topological \wwb{N+k}sphere.
\end{mysubsection}

\begin{example}\label{egfoldp}
The boundaries of the three constituent tetrahedra of \w[,]{\PP{3}{3}} split open,
are illustrated in Figure \ref{fig2}, which also shows how each facet is mapped
under \w[,]{\tHn{j}{3}:\Dels{3}{j}\to\oW{3}} and which facets are identified in
\w{\PP{3}{3}} (dotted arrows).

%
%
\begin{figure}[htbp]
\begin{center}
\begin{picture}(410,165)(10,10)
%
%
\put(70,165){$\partial\Dels{3}{1}$}
%
%
\put(10,10){\circle*{3}}
\put(3,7){\scriptsize $3$}
\put(48,75){\circle*{3}}
\put(40,73){\scriptsize $2$}
\put(85,140){\circle*{3}}
\put(84,145){\scriptsize $3$}
\put(85,10){\circle*{3}}
\put(84,0){\scriptsize $1$}
\put(160,10){\circle*{3}}
\put(164,7){\scriptsize $3$}
\put(122,75){\circle*{3}}
\put(126,73){\scriptsize $0$}
%
%
\put(10,11){\line(1,0){150}}
\put(48,76){\line(1,0){75}}
\bezier{400}(10,10)(48,75)(85,140)
\bezier{400}(160,10)(122,75)(85,140)
\bezier{400}(48,75)(67,43)(85,10)
\bezier{400}(122,75)(104,43)(85,10)
%
%
\put(45,40){\scriptsize $\ast$}
\put(75,46){\scriptsize $\widetilde{\Fk{0} G\sp{0}}$}
\put(75,98){\scriptsize $\widetilde{\iota pH\sp{1}}$}
\put(120,40){\scriptsize $\ast$}
%
%
\bezier{40}(100,90)(125,110)(150,130)
\put(105,94){\vector(-1,-1){7}}
\put(145,125){\vector(1,1){7}}
%
%
\put(200,165){$\partial\Dels{3}{2}$}
%
%
\put(210,10){\circle*{3}}
\put(209,0){\scriptsize $3$}
\put(173,75){\circle*{3}}
\put(166,71){\scriptsize $0$}
\put(135,140){\circle*{3}}
\put(129,141){\scriptsize $3$}
\put(210,140){\circle*{3}}
\put(209,144){\scriptsize $2$}
\put(285,140){\circle*{3}}
\put(288,141){\scriptsize $3$}
\put(248,75){\circle*{3}}
\put(251,71){\scriptsize $1$}
%
%
\put(135,141){\line(1,0){150}}
\put(173,76){\line(1,0){75}}
\bezier{400}(210,10)(173,75)(135,140)
\bezier{400}(210,10)(248,75)(285,140)
\bezier{400}(173,75)(192,107)(210,140)
\bezier{400}(248,75)(229,107)(210,140)
%
%
\put(198,93){\scriptsize $\widetilde{\Fk{1} G\sp{1}}$}
\put(240,110){\scriptsize $\widetilde{\iota pH\sp{2}}$}
\put(165,110){\scriptsize $\widetilde{\iota pH\sp{1}}$}
\put(206,45){\scriptsize $\ast$}
%
%
\bezier{40}(260,135)(290,110)(320,85)
\put(265,130){\vector(-1,1){7}}
\put(315,90){\vector(1,-1){7}}
%
%
\put(330,165){$\partial\Dels{3}{3}$}
%
%
\put(260,10){\circle*{3}}
\put(253,7){\scriptsize $3$}
\put(298,75){\circle*{3}}
\put(291,72){\scriptsize $1$}
\put(335,140){\circle*{3}}
\put(334,145){\scriptsize $3$}
\put(335,10){\circle*{3}}
\put(334,0){\scriptsize $0$}
\put(410,10){\circle*{3}}
\put(414,7){\scriptsize $3$}
\put(372,75){\circle*{3}}
\put(376,72){\scriptsize $2$}
%
%
\put(260,11){\line(1,0){150}}
\put(298,76){\line(1,0){75}}
\bezier{400}(260,10)(298,75)(335,140)
\bezier{400}(410,10)(372,75)(335,140)
\bezier{400}(298,75)(317,43)(335,10)
\bezier{400}(372,75)(354,43)(335,10)
%
%
\put(325,46){\scriptsize $\Fk{2} G\sp{2}$}
\put(325,98){\scriptsize $\widetilde{\iota pH\sp{2}}$}
\put(295,40){\scriptsize $\ast$}
\put(370,40){\scriptsize $\ast$}
\end{picture}
\end{center}

\caption[fig2]{The three tetrahedra of \ $\PP{3}{3}$ \ mapped to $\oW{3}$}
\label{fig2}
\end{figure}
\end{example}

\begin{lemma}\label{lsmash}
Given maps \w{\Fv{N}} and \w{\Gv{N}} as above,
\w{\Phi\sb{(F,G)}:\sms{\partial\PP{N+1}{r}}{\bS{k}}\to\oW{N+1}} is nullhomotopic if
and only if \w{\Psi'\sb{(F,G)}:\hfsm{\partial\PP{N+1}{r}}{\bS{k}}\to\oW{N+1}}
extends to \w[,]{\tHn{}{N}:\hfsm{\PP{N+1}{r}}{\bS{k}}\to\oW{N+1}} implying the
existence of \w{\Hv{N}:\Fv{N}\circ\Gv{N}\sim 0} \wwh and therefore of a lift
of \w{\Gv{n-1}} to \w[.]{\Gv{n}}
\end{lemma}

\begin{proof}
As noted in \S \ref{rfoldp}, the cone \w{C\partial\PP{N+1}{r}} is homeomorphic to
\w[,]{\PP{N+1}{r}} and the quotient map
\w{q:\hfsm{\partial\PP{N+1}{r}}{\bS{k}}\epic\sms{\partial\PP{N+1}{r}}{\bS{k}}} extends
naturally to
\w[.]{q':\hfsm{\PP{N+1}{r}}{\bS{k}}\epic\sms{C\partial\PP{N+1}{r}}{\bS{k}}}
Thus a nullhomotopy for \w{\Phi\sb{(F,G)}} defines an extension of
\w{\Phi'\sb{(F,G)}} to \w[.]{\hHv{N}:\PP{N+1}{r}\times\bS{k}\to\oW{N+1}}
Restricting \w{\tHn{}{N}} to each of the \wwb{N+1}simplices \w{\Dels{N+1}{j}} of
\w{\PP{N+1}{r}} defines a collection of maps \w{\tHn{j}{N}} \wb{j=1,\dotsc,N+1}
satisfying \wref{eqcosimap} (with \w{\iota p H\sp{j}} defined by restricting
\w{\tHn{}{N}} to the appropriate facets gluing the \wwb{N+1}simplices together).
As in the proof of Proposition \ref{pqfib}, the nullhomotopy \w{\Hv{N}} defines
a lift of \w{\Gv{N}} to \w[.]{\hGv{N+1}:\hfsm{\Deu}{\bS{k}}\to\hWu{N+1}}

If we assume by induction that \w{\Gn{i}{N}:\hfsm{\Del{i}}{\bS{k}}\to\Wn{i}{N}}
is zero for \w[,]{i<n} we may assume the same for the maps \w[,]{\Hn{i}{N}}
and thus for \w[.]{\hGn{i}{N+1}}
We then use the LLP in the Reedy model category of cosimplicial spaces to obtain the
required lift:
\mydiagram[\label{eqllp}]{
\hfsm{\sk{n-1}\Deu}{\bS{k}} \ar[rrr]\sp{0} \ar@{^{(}->}[d]\sb{\inc} &&&
\W{N+1} \ar@{->>}[d]\sp{q\bp{N+1}}\sb{\simeq} & \\
\hfsm{\Deu}{\bS{k}} \ar[rrr]\sb{\hGv{N+1}} \ar@{.>}[urrr]\sp{\Gv{N+1}} &&& \hWu{N+1},
}
\noindent again with \w{\Gn{i}{N+1}=0} for \w[.]{i<n}
\end{proof}

\begin{defn}\label{dhco}
Assume given an ($R$-good) space $\bY$, a CW resolution \w{\Vd\in s\RAlg} of
\w[,]{\HsR{\bY}} and a realization \w{\Wu} of \w{\Vd} constructed by
suitable choices of maps \w{\Fv{n}:\W{n}\to\Du{n}} for each \w[,]{n\geq 0}
as in \S \ref{srcwr}. For each pair \w{(k,n)} we then have a sequence of higher 
cohomology operations
$$
\llrr{-}{r}~:~\pi\sb{k+n}\oW{n}~\to~\pi\sb{N+k}\oW{N+1}~=~\pi\sb{k+n+r-1}\oW{n+r}
$$
\noindent for \w{r\geq 1} and \w[,]{N=n+r-1} which serve as obstructions to
lifting \w{\gamma\in\pi\sb{k}\Omega\sp{n}\oW{n}} to
\w[,]{\Gv{n+r+1}:\hfsm{\Deu}{\bS{k}}\to\W{n+r+1}} The \emph{data} for
\w{\llrr{-}{r}} consists of compatible choices of lifts
\w{\Gv{i}:\hfsm{\Deu}{\bS{k}}\to\W{i}} for \w{n\leq i\leq N} as
above, and the \emph{value} of \w{\llrr{\gamma}{r}} associated to this data is
the class of \w{[\Phi\sb{(F,G)}]\in\pi\sb{N+k}\oW{N+1}}
constructed in \S \ref{sindcon}. The \emph{indeterminacy} of \w{\llrr{\gamma}{r}} is
the subset of \w{\pi\sb{N+k}\oW{N+1}} consisting of all possible values, for all
(compatible) choices of the data \w{\Gv{i}} (with the maps
\w{(\Fv{n}:\W{n}\to\Du{n})\sb{n=0}\sp{\infty}} fixed). We say that the operation
\emph{vanishes} if there is a choice of the data \w{\Gv{i}} with value zero.
\end{defn}

From the description above we deduce:

\begin{thm}\label{thco}
Each value \w{\llrr{\gamma}{r}\in\pi\sb{k+n+r-1}\oW{n+r}=E\sb{1}\sp{n+r,k+n+r-1}}
of the $r$-th order operation \w{\llrr{-}{r}} represents the result of applying
the differential \w{d\sb{r}} to the element of \w{E\sb{r}\sp{n,k+n}} represented by
\w[.]{\gamma\in\pi\sb{k+n}\oW{n}=E\sb{1}\sp{n,k+n}}
The indeterminacy for \w{\llrr{-}{r}} is the same as that for the differential (as
a map \w[).]{E\sb{1}\to E\sb{1}}
\end{thm}

\begin{remark}\label{rhco}
More generally, we can think of the maps \w{\Fv{n+r}:\W{n+r}\to\Du{n+r}} of 
\S \ref{srcwr}(d) as providing a template for a system of higher operations
$$
\llrr{-}{r}:[\bZ,\,\Omega\sp{n}\oW{n}]~\to~[\bZ,\,\Omega\sp{n+r-1}\oW{n+r}]
$$
\noindent for \w{r=2,3,\dotsc} (here we had \w[).]{\bZ=\bS{k}} By a ``template'' we 
mean that we can fix the maps \w{\Fv{i}} once and for all, while the choices of 
the liftings \w{\Gv{n+r}} may need to be changed if a particular value of 
\w{\llrr{\gamma}{r}} is non-zero. 

From Theorem \ref{thco}, and the fact that the unstable Adams spectral sequence is
unique up to isomorphism, it follows that \emph{any} choice of this template
yields equivalent operations, in the sense that vanishing for one choice implies their
vanishing for any other choice of the maps \w{\Fv{i}} \wwh for the right choice
of data \w[.]{\Gv{i}}
\end{remark}

%
%
\sect{Filtration in the unstable Adams spectral sequence}
\label{cfuass}

We now consider the question of determining the filtration index of a non-zero
element in the unstable Adams spectral sequence.
Here we assume that the coaugmented cosimplicial
space \w{\bY\to\Wu} is constructed as in \S \ref{crcwr} for an $R$-good space $\bY$,
with $R$-completion \w[,]{\hY\simeq\Tot\Wu} so we require that
\begin{myeq}\label{eqpitower}
\pi\sb{k}\Tot\Wu~=~\lim\sb{n}\ \pi\sb{k}\Tot\W{n}~.
\end{myeq}
\noindent If we write \w{p\bp{n}:\Tot\Wu\to\Tot\sb{n}\W{n}} for the appropriate
structure map in \wref[,]{eqmodtott} the \emph{filtration index} for
\w{0\neq\gamma\in\pi\sb{k}\Tot\Wu} is the least $n$ for which
\w[.]{(p\bp{n})\sb{\ast}\gamma\neq 0} Thus
\w[,]{(\prn{n}\circ p\bp{n})\sb{\ast}\gamma=(p\bp{n-1})\sb{\ast}\gamma=0} so
\w{(p\bp{n})\sb{\ast}\gamma} lifts to some element
\w[,]{\alpha\in\pi\sb{k}\Omega\sp{n}\oW{n}=E\sb{1}\sp{n,n-k}}
by Proposition \ref{pqfib}. This $\alpha$ represents $\gamma$ in the spectral
sequence.

In general, $\gamma$ is represented by a map of cosimplicial
spaces \w[,]{\Gamma:\hfsm{\Deu}{\bS{k}}\to\Wu} as in \S
\ref{stott}. However, because $\hY$ is $R$-good, the
$R$-completion \w{\hY\simeq\Tot\Wu} is coaugmented to \w{\Wu} by
Remark \ref{rhcocomp}. Therefore (replacing $\bY$ by $\hY$, if
necessary), we may assume for simplicity that $\gamma$ is
represented by a map \w[.]{g:\bS{k}\to\bY} The map $\Gamma$ thus
factors in cosimplicial dimension $j$ as the composite
\w{\Gamma\sp{j}} of
\begin{myeq}\label{eqsimplecomp}
\hfsm{\Del{j}}{\bS{k}}~\xepic{q}~\bS{k}~\xra{g}~\bY~\xra{D\sp{j}}~\bW\sp{j}~,
\end{myeq}
\noindent where \w{D\sp{j}} is the unique iterated coface map starting from
the coaugmentation, say:
\begin{myeq}\label{eqitcof}
D\sp{j}~:=~d\sp{j}\circ\dotsc d\sp{1}\circ\vare~.
\end{myeq}
\noindent For \w[,]{\Wu=\hWu{n}} we denote the projection of \w{D\sp{j}}
onto the factor \w{P\Omega\sp{n-j-1}\oW{n}} by \w[.]{\oD{j}=\oD{j}\bp{n}}

\begin{mysubsection}{The inductive process}\label{sindfil}
We start with the map \w{\Gamma\sp{0}} given by
\w[.]{\bve{0}\circ g:\bS{k}\to\Wn{0}{0}} If this is non-zero, $g$ has filtration
index $0$ (which means it is
``visible to $R$-cohomology'', since \w{\bve{0}} encodes all cohomology classes of
$\bY$). Otherwise, we have a nullhomotopic map of cosimplicial spaces
\w{\Gamma\bp{0}:\hfsm{\Deu}{\bS{k}}\to\W{0}} (of the simple kind given by
\wref[),]{eqsimplecomp} and we can choose a nullhomotopy \w{\Gv{0}} for it.

In the induction step, until we reach the filtration index, assume
we have \w{\Gamma\bp{n-1}:\hfsm{\Deu}{\bS{k}}\to\W{n-1}} as in
\wref[,]{eqsimplecomp} with a nullhomotopy
\w[.]{\Gv{n-1}:\sms{C\Deu}{\bS{k}}\to\W{n-1}} The $0$-th coface
maps \w{\delta\sp{0}:\Del{j}\hra\Del{j+1}=C\Del{j}} fit together
to define a map of cosimplicial spaces \w[,]{\inc:\Deu\hra C\Deu}
with \w[.]{\Gv{n-1}\circ\inc=\Gamma\bp{n-1}}

From \wref{eqsmatch} we see that to lift \w{\Gv{n-1}} to \w[,]{\hWu{n}}
for each \w{0<j\leq n} we must choose maps
\w{H\sp{j}=H\sp{j}\bp{n}:\sms{C\Del{j}}{\bS{k}}\to P\Omega\sp{n-j-1}\oW{n}}
making the following diagram commute:

\myudiag[\label{eqlnulhtpy}]{
\sms{C\Del{j}}{\bS{k}} \ar@/^{1.5pc}/[rrrrrd]\sp(0,6){G\sp{j}}
\ar@/^{2.5pc}/[rrrrrrrd]\sp(0.6){H\sp{j}} &&&&&&& \\
& \hfsm{\Del{j}}{\bS{k}} \ar@{_{(}->}[ul]\sb{\delta\sp{0}=\inc}
\ar@{->>}[r]\sp(0.6){q} & \bS{k} \ar[r]\sp{g} &
\bY\ar[rr]\sp(0.35){D\sp{j}} \ar@/_{2.1pc}/[rrrr]\sb(0.6){\oD{j}} &&
\Wn{j}{n-1}  & \times & P\Omega\sp{n-j-1}\oW{n} \\
\sms{C\Del{j-1}}{\bS{k}}
\ar@<1.5ex>[uu]\sp(0.5){\delta\sp{1}=}\sp(0.4){Cd\sp{0}}\sb(0.45){\dotsc}
 \ar@<-1.5ex>[uu]\sb(0.5){\delta\sp{j+1}}\sb(0.4){=Cd\sp{j}}
\ar@/^{2.0pc}/[rrrrrd]\sp(0.6){G\sp{j-1}}
\ar@/_{6pc}/[rrrrrrrd]\sp(0.5){H\sp{j-1}} &&&&&&&\\
& \hfsm{\Del{j-1}}{\bS{k}} \ar@{_{(}->}[ul]\sb{\delta\sp{0}=\inc}
\ar@<1.5ex>[uu]\sp(0.7){d\sp{0}}\sb(0.7){\dotsc}
\ar@<-1.5ex>[uu]\sb(0.7){d\sp{j}} \ar@{->>}[r]\sp(0.7){q} & \bS{k}
\ar[r]\sp{g} & \bY\ar[rr]\sp(0.35){D\sp{j-1}}
\ar@/_{1.6pc}/[rrrr]\sb{\oD{j-1}} && \Wn{j-1}{n-1}
\ar@<1.5ex>[uu]\sp(0.6){d\sp{0}}\sb(0.6){\dotsc}
\ar@<-1.5ex>[uu]\sb(0.6){d\sp{j}} \ar[rruu]\sb{d\sp{0}=\Fk{j-1}}
& \times & P\Omega\sp{n-j}\oW{n}
 \ar[uu]\sb(0.45){d\sp{1}=\iota p}
}

\vs\quad

\noindent where \w{q:\hfsm{\Del{j}}{\bS{k}}\to\bS{k}} is the projection.
Note that we have chosen a different indexing for the coface maps
under the identification of \w{C\Del{j}} with \w[.]{C\Del{j+1}}

The maps \w{H\sp{j}} must satisfy:
\begin{myeq}\label{eqlnullhtpy}
\begin{split}
H\sp{j}\circ Cd\sp{0}~=&~\Fk{j-1}\circ G\sp{j-1},\
H\sp{j}\circ Cd\sp{1}~=~\iota\circ p\circ H\sp{j-1},\\
H\sp{j}\rest{\hfsm{\Del{j}}{\bS{k}}}~=&~\oD{j}\circ g\circ q, \ \
 \text{and}\ H\sp{k}\circ Cd\sp{i}~=~0 \ \text{for} \ i\geq 2~,
\end{split}
\end{myeq}
\noindent for \w[,]{0< j\leq n} as in \wref[.]{eqnullhtpy}

Using the conventions of \S \ref{sindcon}, we let
\w{\tH{j}=\tH{j}\bp{n}:\Del{n+1}\to\oW{n}} denote the appropriate adjoint of
\w[,]{H\sp{j}=H\sp{j}\bp{n}} and deduce from \wref{eqlnullhtpy} that for each
\w[:]{0\leq j\leq n}
\begin{myeq}\label{eqcosimpmap}
\tH{j}\circ\delta\sp{i}~=~\begin{cases}
\widetilde{\oD{j}gq} & \text{for}\hsm i=0\\
\widetilde{\iota p H\sp{j}} & \text{for}\hsm i=n-j\hsm \text{and}\hsm j\leq n-1\\
\widetilde{\Fk{j-1}G\sp{j-1}}  & \text{for}\hsm i=n-j+1\hsm
\text{and}\hsm j>0\\
\widetilde{\iota p H\sp{j-1}} & \text{for}\hsm i=n-j+2\hsm \text{and}\hsm j>0\\
0  & \text{otherwise,}
\end{cases}
\end{myeq}
\noindent as in \wref[.]{eqcosimap}

For \w{j=0} we may let \w{H\sp{0}:C\bS{k}\to P\Omega\sp{n-1}\oW{n}} be the
tautological nullhomotopy of
\w{\Fk{-1}\vare g:\bS{k}\to P\Omega\sp{n-1}\oW{n}}
(so that its adjoint \w[,]{\tH{0}:\Del{n+1}\to\oW{n}} depicted on the left in
Figure \ref{fig3}, has $2$-facet
\w[).]{\widetilde{\iota pH\sp{0}}=\widetilde{\Fk{-1}\vare g}}

As before, we use \wref{eqcosimpmap} to glue the maps \w{\tH{j}}
and define a map \w{\tH{}:\hfsm{\hP{n+1}{n+2}}{\bS{k}}\to\oW{n}}
(see Remark \ref{rnewfold}). Its restriction to
\w{\hfsm{\partial\hP{n+1}{n+2}}{\bS{k}}} again depends only on
\w{\bF=\Fv{n}} and the chosen nullhomotopy \w[,]{\bG=\Gv{n-1}} so
we may denote it by
\w[.]{\Psi'\sb{(F,G)}:\hfsm{\partial\hP{n+1}{n+2}}{\bS{k}}\to\oW{n}}
Moreover, \w{\Psi'\sb{(F,G)}} is zero on \w{\{v\}\times\bS{k}} for
any cone vertex $v$ of \w[,]{\hP{n+1}{n}} so it induces a map
\w{\Psi\sb{(F,G)}:\sms{\partial\hP{n+1}{n+2}}{\bS{k}}\to\oW{n}}
from a PL \wwb{n+k}sphere.
\end{mysubsection}

\begin{remark}\label{ritcoface}
If \w[,]{j\geq 2} the projection of \w{d\sp{j}} onto
\w{P\Omega\sp{n-j-1}\oW{n}} is zero, so \w{\oD{j}\bp{n}=0} by
\wref[,]{eqitcof} and thus \w{\tH{j}\bp{n}\circ\delta\sp{n+1}=0}
by \wref[.]{eqcosimpmap} On the other hand,
\w[,]{\tH{1}\bp{n}\circ\delta\sp{n+1}=\widetilde{\Fk{0}\circ\vare\circ
g\circ q}} and
\w[.]{\tH{0}\bp{n}\circ\delta\sp{n+1}=\widetilde{\Fk{-1}\circ g}}
Note that it is consistent with the description in \S \ref{srcwr} to think of
\w{g:\bS{k}\to\bY} as \w[,]{G\sp{-1}} so more suggestively we may write this
last as \w[.]{\widetilde{\Fk{-1}\circ G\sp{-1}}}
\end{remark}

\begin{example}\label{egfldp}
The boundaries of the three constituent tetrahedra of
\w{\hP{3}{3}} are given in Figure \ref{fig3}, including their
identifications, showing how the facets are mapped to \w{\oW{2}}
under \w[.]{\tH{j}\bp{2}}

%
%
\begin{figure}[htbp]
\begin{center}
\begin{picture}(410,160)(10,10)
%
%
\put(70,165){$\partial\Dels{3}{0}$}
%
%
\put(10,10){\circle*{3}}
\put(3,7){\scriptsize $3$}
\put(48,75){\circle*{3}}
\put(40,73){\scriptsize $1$}
\put(85,140){\circle*{3}}
\put(84,145){\scriptsize $3$}
\put(85,10){\circle*{3}}
\put(84,0){\scriptsize $2$}
\put(160,10){\circle*{3}}
\put(164,7){\scriptsize $3$}
\put(122,75){\circle*{3}}
\put(126,73){\scriptsize $0$}
%
%
\put(10,11){\line(1,0){150}}
\put(48,76){\line(1,0){75}}
\bezier{400}(10,10)(48,75)(85,140)
\bezier{400}(160,10)(122,75)(85,140)
\bezier{400}(48,75)(67,43)(85,10)
\bezier{400}(122,75)(104,43)(85,10)
%
%
\put(81,48){\scriptsize $\ast$}
\put(75,98){\scriptsize $\widetilde{\iota pH\sp{0}}$}
\put(30,26){\scriptsize $\widetilde{\Fk{-1}G\sp{-1}}$}
\put(121,32){\scriptsize $\ast$}
%
%
\bezier{40}(100,90)(125,110)(150,130)
\put(105,94){\vector(-1,-1){7}}
\put(145,125){\vector(1,1){7}}
%
%
\put(200,165){$\partial\Dels{3}{1}$}
%
%
\put(210,10){\circle*{3}}
\put(209,0){\scriptsize $2$}
\put(173,75){\circle*{3}}
\put(166,71){\scriptsize $1$}
\put(135,140){\circle*{3}}
\put(129,141){\scriptsize $2$}
\put(210,140){\circle*{3}}
\put(209,144){\scriptsize $0$}
\put(285,140){\circle*{3}}
\put(288,141){\scriptsize $2$}
\put(248,75){\circle*{3}}
\put(251,71){\scriptsize $3$}
%
%
\put(135,141){\line(1,0){150}}
\put(173,76){\line(1,0){75}}
\bezier{400}(210,10)(173,75)(135,140)
\bezier{400}(210,10)(248,75)(285,140)
\bezier{400}(173,75)(192,107)(210,140)
\bezier{400}(248,75)(229,107)(210,140)
%
%
\put(198,93){\scriptsize $\widetilde{\Fk{0} G\sp{0}}$}
\put(240,110){\scriptsize $\iota pH\sp{1}$}
\put(165,110){\scriptsize $\widetilde{\iota pH\sp{0}}$}
\put(198,46){\scriptsize $\widetilde{\Fk{0}\vare gq}$}
%
%
\bezier{40}(260,135)(290,110)(320,85)
\put(265,130){\vector(-1,1){7}}
\put(315,90){\vector(1,-1){7}}
%
%
\put(330,165){$\partial\Dels{3}{2}$}
%
%
\put(260,10){\circle*{3}}
\put(253,7){\scriptsize $1$}
\put(298,75){\circle*{3}}
\put(291,72){\scriptsize $0$}
\put(335,140){\circle*{3}}
\put(334,145){\scriptsize $1$}
\put(335,10){\circle*{3}}
\put(334,0){\scriptsize $2$}
\put(410,10){\circle*{3}}
\put(414,7){\scriptsize $1$}
\put(372,75){\circle*{3}}
\put(376,72){\scriptsize $3$}
%
%
\put(260,11){\line(1,0){150}}
\put(298,76){\line(1,0){75}}
\bezier{400}(260,10)(298,75)(335,140)
\bezier{400}(410,10)(372,75)(335,140)
\bezier{400}(298,75)(317,43)(335,10)
\bezier{400}(372,75)(354,43)(335,10)
%
%
\put(325,48){\scriptsize $\Fk{1} G\sp{1}$}
\put(325,98){\scriptsize $\iota pH\sp{1}$}
\put(296,31){\scriptsize $\ast$}
\put(352,22){\scriptsize $\oD{2}gq=\ast$}
\end{picture}
\end{center}

\caption[fig3]{The three tetrahedra of \ $\hP{3}{3}$, \ mapping to $\oW{3}$}
\label{fig3}
\end{figure}

Note that all $2$-simplices of the boundary \w{\partial\hP{3}{3}} map to \w{\oW{3}}
by zero or by a map of the form \w{\widetilde{\Fk{j} G\sp{j}}} \wb[,]{-1\leq j\leq 1}
except for one mapping by \w[.]{\widetilde{\Fk{0}\vare gq}} However, the
fact that we precompose this map with the projection
\w{q:\hfsm{\Del{1}}{\bS{k}}\to\bS{k}} indicates that this $2$-simplex in
\w{\partial\hP{3}{3}} is \emph{degenerate} (that is, it may be collapsed to
its top edge, the $0$-face of \w[).]{\widetilde{\Fk{0} G\sp{0}}}
\end{example}

As in Lemma \ref{lsmash} we deduce:

\begin{lemma}\label{lsmasht}
Given maps \w{\Fv{n}} and a nullhomotopy \w{\Gv{n-1}} as above,
\w{\Psi\sb{(F,G)}:\sms{\partial\PP{n+1}{n+1}}{\bS{k}}\to\oW{n}} is nullhomotopic if
and only if \w{\Psi'\sb{(F,G)}:\hfsm{\partial\PP{N+1}{r}}{\bS{k}}\to\oW{N+1}}
extends to \w[,]{\hHv{N}:\hfsm{\PP{N+1}{r}}{\bS{k}}\to\oW{N+1}} implying existence
of a nullhomotopy \w{\Hv{n}} \wwh and therefore of a lift of \w{\Gv{n-1}}
to \w[.]{\Gv{n}}
\end{lemma}

\begin{defn}\label{dfhco}
By analogy with \S \ref{dhco}, given a space $\bY$, a CW resolution
\w{\Vd\in s\RAlg} of \w[,]{\HsR{\bY}} and maps \w{\Fv{n}:\W{n}\to\Du{n}} 
\wb[,]{n\geq 0} yielding a realization \w{\Wu} of \w{\Vd} as in \S \ref{srcwr}, 
we have a new sequence of higher cohomology operations
\w{\llrrp{-}{r}:\pi\sb{k}\bY\to\pi\sb{k+n}\oW{n}} for each pair \w[,]{(k,n)}
which serve as obstructions to representing \w{\gamma\in\pi\sb{k}\bY} in
higher Adams filtration by lifting nullhomotopies \w{\Gv{i}} \wb{0\leq i<n}
\wwh the \emph{data} for \w{\llrrp{-}{r}} \wwh
to a nullhomotopy \w[.]{\Gv{n}}

The \emph{value} of \w{\llrrp{\gamma}{r}} associated to this data is
the class of \w{[\Psi\sb{(F,G)}]\in\pi\sb{n+k}\oW{n}} constructed in
\S \ref{sindfil}. The \emph{indeterminacy} of \w{\llrrp{\gamma}{r}} is
the subset of \w{\pi\sb{n+k}\oW{n}} consisting of all possible values, for all
(compatible) choices of the data \w{\Gv{i}} (with the maps
\w{(\Fv{n}:\W{n}\to\Du{n})\sb{n=0}\sp{\infty}} again fixed).
\end{defn}

Remark \ref{rhco} as to the independence of the operations \w{\llrrp{-}{r}}
from the ``template'' maps \w{\Fv{n}:\W{n}\to\Du{n}} applies here too,
\emph{mutatis mutandis}, because of the following:

\begin{thm}\label{tfil}
For any \w[,]{0\neq\gamma\in\pi\sb{k}\bY} the operation \w{\llrrp{\gamma}{n-1}}
vanishes while \w{\llrrp{\gamma}{n}\neq 0} if and only if $\gamma$ is represented
in the unstable Adams spectral sequence in filtration $n$ by the value of
\w{\llrrp{\gamma}{n}} in \w[.]{\pi\sb{k}\Omega\sp{n}\oW{n}=E\sb{1}\sp{n,n-k}}
\end{thm}

\begin{proof}
By  Lemma \ref{lsmasht} \w{\llrrp{\gamma}{n-1}} vanishes (for some choice of
nullhomotopy \w[)]{\Gv{n-1}} if and only if we have a lift of \w{\Gv{n-1}} to
a nullhomotopy \w[,]{\Gv{n}} so $\gamma$ has Adams filtration $\geq n$.
Assume \w{\llrrp{\gamma}{n}} is represented by
\w[,]{\Psi\sb{(F,G)}:\sms{\partial\hP{n+1}{n}}{\bS{k}}\to\oW{n}} induced by
\w{\tH{}:\hfsm{\hP{n+1}{n}}{\bS{k}}\to\oW{n}} as in \S \ref{sindfil}.

For each \w{0\leq j\leq n} denote \w{\Fk{j-1}\circ G\sp{j-1}} by
\w[.]{\phi\sp{j}:\hfsm{\Del{j}}{\bS{k}}\to P\Omega\sp{n-j-1}\oW{n}}
For \w{j=0} we let \w{\sms{C\Del{-1}}{\bS{k}}:=\bS{k}} and \w{\Wn{-1}{n-1}:=\bY}
(compare \S \ref{srcwr}(b)), with \w{G\sp{-1}:\sms{C\Del{-1}}{\bS{k}}\to\Wn{-1}{n-1}}
equal to $g$ as in \S \ref{ritcoface}.

Using the conventions of \S \ref{sindfil} for \w[,]{j\geq 1} we write
\w{\delta\sp{0}}  for the inclusion of the base
\w[,]{\hfsm{\Del{j}}{\bS{k}}\hra\sms{C\Del{j}}{\bS{k}}} and \w{\delta\sp{i+1}} for
\w{Cd\sp{i}:\sms{C\Del{j-1}}{\bS{k}}:\to\sms{C\Del{j}}{\bS{k}}}
\wb[.]{0\leq i\leq j}
From \wref{eqlnulhtpy} and \wref{eqlnullhtpy} we see that for each
\w{j\geq 1} and \w[,]{0\leq i\leq j} we have
\begin{myeq}\label{eqequals}
\phi\sp{j+1}\circ \delta\sp{i+1}:=\Fk{j}\circ G\sp{j}\circ Cd\sp{i}=
\Fk{j}\circ d\sp{i}\circ G\sp{j-1}=
d\sp{0}\circ d\sp{i}\circ G\sp{j-1}=d\sp{i+1}\circ d\sp{0}\circ G\sp{j-1}
\end{myeq}
\noindent into \w[,]{P\Omega\sp{n-j-2}\oW{n}} which is
\w{d\sp{1}\circ\Fk{j-1}\circ G\sp{j-1}=d\sp{1}\circ\phi\sp{j}} when \w[.]{i=0}
When \w[,]{i\geq 1} the projection of \w{d\sp{i+1}} onto \w{P\Omega\sp{n-j-2}\oW{n}}
vanishes (cf.\ \S \ref{srcwr}), so by \wref{eqequals} both
\w{\phi\sp{j+1}\circ \delta\sp{i+1}} and \w{d\sp{i+1}\circ\phi\sp{j}} are zero there.

Similarly, by \wref{eqitcof} we have
\begin{myeq}\label{eqcosidea}
\phi\sp{j+1}\circ \delta\sp{0}=\Fk{j}\circ G\sp{j}\rest{\hfsm{\Del{j}}{\bS{k}}}
=\Fk{j}\circ D\sp{j}\circ g=d\sp{0}d\sp{j}\dotsc\vare\circ g~=
d\sp{j+1}d\sp{j}\dotsc\vare\circ g~=~0
\end{myeq}
\noindent into \w[,]{P\Omega\sp{n-j-2}\oW{n}} and also
\w{d\sp{0}\rest{P\Omega\sp{n-j-1}\oW{n}}} is zero.

Finally, for \w{j=0} we write \w{\delta\sp{1}} for the inclusion of the base
\w{\bS{k}} into \w[,]{C\bS{k}=\sms{\Del{1}}{\bS{k}}} and \w{\delta\sp{0}} for the
map collapsing \w{\bS{k}} to the cone point, so we have:
\begin{myeq}\label{eqcosideb}
\phi\sp{1}\circ \delta\sp{i}:=\Fk{0}\circ G\sp{0}\circ \delta\sp{i}=
\begin{cases}
\iota p\circ\Fk{-1}\circ G\sp{-1}=d\sp{1}\phi\sp{0}& \text{if}\ i=1\\
0=d\sp{0}\phi\sp{0}& \text{if}\ i=0\\
\end{cases}
\end{myeq}
\noindent into \w[.]{P\Omega\sp{n-2}\oW{n}}

Thus for each \w{j\geq 0} we have:
\begin{myeq}\label{eqcoside}
d\sp{i}\circ\phi\sp{j}~=~\phi\sp{j+1}\circ \delta\sp{i} \hsp\text{for all} \
0\leq i\leq j+1
\end{myeq}
\noindent (into \w[).]{P\Omega\sp{n-j-2}\oW{n}}

Therefore, by the proof of Lemma \ref{ldtot},
the maps \w{\phi\sp{j}:\hfsm{\Del{j}}{\bS{k}}\to P\Omega\sp{n-j-1}\oW{n}}
uniquely determine a pointed map \w[.]{\phi:\bS{k}\to\Tot\sDu{n}}
The inclusion \w{i\bp{n}:\sDu{n}\hra\hWu{n}} allows us to think of $\phi$ as
a map \w[,]{(i\bp{n})\sb{\ast}\phi=\varphi:\bS{k}\to\Tot\sb{n}\hWu{n}} which maps
by zero to the factor \w{\Wn{j}{n-1}} of \w{\hWn{j}{n}} in \wref[.]{eqcindstep}

Let \w{t\sp{j}:\Del{j}\times I\epic C\Del{j}\cong\Del{j+1}} be the map
collapsing the top of the prism, \w[,]{\Del{j}\times\{1\}} to a point,
and let \w{\sigma\sp{0}:\Del{j+1}\epic\Del{j}} be the $0$-the codegeneracy
map of \w[.]{\Deu}

We then define a pointed homotopy \w{K:\hfsm{I}{\bS{k}}\to\Tot\sb{n}\hWu{n}}
by setting \w{K\sp{j}:\hfsm{(\Del{j}\times I)}{\bS{k}}\to\hWn{j}{n}} equal to:
\myrdiag[\label{eqhtpy}]{
\hfsm{(\Del{j}\times I)}{\bS{k}}
\ar[rrrrr]\sp(0.42){(G\sp{j}\circ t\sp{j})\top(X\sp{j})\top
(\Fk{j-1}\circ G\sp{j-1}\circ\sigma\sp{0}\circ t\sp{j})}
&&&&& \Wn{j}{n-1}\times \hM{r-1}{n}\times P\Omega\sp{n-j-1}\oW{n}
}
\noindent for each \w[,]{1\leq j\leq n} where \w{X\sp{j}} is determined by the
codegeneracies on \w{G\sp{j}} and
\w[.]{\Fk{j-1}\circ G\sp{j-1}\circ\sigma\sp{0}\circ t\sp{j}}
When \w{j=0} we have \w{\Gamma\bp{n}\sp{0}=\varphi\sp{0}} into
\w[,]{P\Omega\sp{n-1}\oW{n}} with the third map the identity homotopy.

The restriction \w{K\circ\iota\sb{0}} to the $0$-end of $I$ (the base of each
prism \w[)]{\Del{j}\times I}
is the given map \w[,]{\Gamma\bp{n}:\hfsm{C\Deu}{\bS{k}}\to\hWu{n}} since
\w{K\sp{j}\circ\iota\sb{0}=D\sp{j}\circ g\circ q:\hfsm{\Del{j}}{\bS{k}}\to\hWn{j}{n}}
by \wref[.]{eqlnulhtpy}

On the other hand, \w{K\circ\iota\sb{1}} is zero into each factor \w[,]{\Wn{j}{n-1}}
since \w{G\sp{j}} is a nullhomotopy, while the component into
\w{P\Omega\sp{n-j-1}\oW{n}} is \w[.]{\Fk{j-1}\circ G\sp{j-1}~=~\phi\sp{j}}
Thus \w{K\circ\iota\sb{1}} is \w[,]{\varphi:\bS{k}\to\Tot\sb{n}\hWu{n}}
showing that \w{[\phi]\in\pi\sb{k}\Tot\sDu{n}} (the value of
\w[)]{\llrrp{\gamma}{n-1}} indeed represents \w{\Gamma\bp{n}}
(the lift of $\gamma$ to the $n$-th level in \wref[).]{eqmodtott}
\end{proof}

\begin{mysubsection}{The one-dimensional case}\label{egdifone}
From the description in \S \ref{srcwr} we have
\mydiagram[\label{eqwone}]{
\bY \ar[d]\sb{\bve{1}} &=&&& \bY\ar[dll]\sb{\bve{0}} \ar[drr]\sp{\Fk{-1}}&&\\
\Wn{0}{1} \ar@/_{1.5pc}/[d]\sb{\dz{0}} \ar@/^{1.5pc}/[d]\sp{d\sp{1}\sb{0}} &=&
\oW{0} \ar@/_{0.5pc}/[d]\sb{\dz{0}=d\sp{1}\sb{0}=\Id}
\ar[drr]\sp{\udz{0}} && \times &&
P\oW{1} \ar[dll]\sb{d\sp{1}\sb{0}=p}
\ar@/^{0.5pc}/[d]\sp{\dz{0}=d\sp{1}\sb{0}=\Id} \\
\Wn{1}{1} \ar[u]\sp{s\sp{0}} &=& \oW{0}\ar@/_{0.5pc}/[u]\sb{=} &\times &
\oW{1} & \times & P\oW{1}.\ar@/^{0.5pc}/[u]\sp{=}
}
\noindent where $p$ is the path fibration, so \w{\Fk{-1}} is a nullhomotopy for
\w[,]{a\sp{-1}:=\udz{0}\circ\bve{0}} in the notation of \wref[.]{eqcochainmap}

If \w{\gamma\in\pi\sb{i}\hY} has filtration index $\geq 1$, it can be represented by
a map \w{g:\bS{i}\to\bY} with \w[.]{\bve{0}\circ g\sim\ast} Choosing a
nullhomotopy \w{G:\bS{i}\to P\oW{0}} (with \w[),]{p\circ G=\bve{0}\circ g}
we obtain a solid commutating diagram
\mydiagram[\label{eqfone}]{
&&& P\oW{0} \ar[rr]\sp{P\udz{0}} \ar[dr]\sp{p} && P\oW{1} \ar[dr]\sp{p} & \\
\bS{i} \ar@/^1.5pc/[rrru]\sp{G} \ar[d]\sb{g} \ar@{-->}[rr]\sp{\vartheta} &&
\Omega\oW{1} \ar[ru] \ar[rd] && \oW{0} \ar[rr]\sp{\udz{0}} && \oW{1}\\
\bY \ar[rrr]\sb{\Fk{-1}} \ar[rrrru]\sp(0.7){\bve{0}} |!{[rrr];[uu]}\hole&&&
P\oW{1} \ar@/_0.5pc/[rrru]\sb{p} &&&
}
\noindent where the outer right pentagon is Cartesian, thus defining $\vartheta$
into the pullback, which is a (non-standard model for) \w[.]{\Omega\oW{1}}
In fact, \w{[\vartheta]} is (one value of) the Toda bracket
\w[,]{\lra{\udz{0},\,\bve{0},\,g}\subseteq\pi\sb{i+1}\oW{1}} and it is non-zero
if and only if $g$ has filtration $1$, since \w[,]{\Omega\oW{1}=\Omega N\sp{1}\W{1}}
and \w{j\sp{1}[\vartheta]} is the lift of $g$ to
\w{\Tot\sb{1}\Wu\simeq\Tot\sb{1}\W{1}} (cf.\ \S \ref{sucws}) in Figure \ref{fig1}.

Combining this with Theorem \ref{tfil}, we see that $\gamma$ has filtration
index $1$ if and only if there is a secondary cohomology operation which acts on it
non-trivially (that is, cannot be made to vanish for any choice of nullhomotopies).

More generally, $\gamma$ has filtration index $n$ if and only if there is an
\wwb{n+1}th order cohomology operation which acts on it non-trivially (where
$n$ is necessarily the highest order for which this is possible).
\end{mysubsection}

\begin{example}\label{eqsecopn}
Consider the cofibration sequence
\w{\bS{n}\xra{2}\bS{n}\xra{r}\bS{n}\cup\sb{2}\be{n+1}\xra{\nabla}\bS{n+1}}
for the mod $2$ Moore space, fitting into the homotopy commutative diagram
\myydiag[\label{eqmooresp}]{
\bS{n}\ar[r]\sp{2}& \bS{n} \ar[r]\sp(0.35){\inc} \ar[rrd]\sb{\iota\sb{n}} &
\bS{n}\cup\sb{2}\be{n+1} \ar[r]\sp{\nabla} \ar[rd]\sp{p} &
\bS{n+1} \ar[r]\sp(0.4){q} & \KZ{n+1} \ar[d]\sp{\rho\sb{2}} \\
&&& \KP{\Ft}{n} \ar[r]\sp(0.4){\Sq{1}} & \KP{\Ft}{n+1}
}
\noindent where $p$ and $q$ are the appropriate Postnikov fibrations and
\w{\rho\sb{2}} is the reduction mod $2$ map.

Therefore, we can realize a free simplicial resolution of the \TRal
\w{\Hus{\bS{n}}{\Ft}} by \w{\bY=\bS{n}\to\Wu} in cosimplicial dimensions 
\w[:]{\leq 1}
\begin{myeq}\label{eqressn}
\bS{n}~\xra{\iota\sb{n}}~\KP{\Ft}{n}~\xra{\Sq{1}}~\KP{\Ft}{n+1}
\end{myeq}
\noindent For simplicity we have omitted all factors \w{\KP{\Ft}{i}}
for \w[.]{i>n+1} Assuming that \w[,]{n\geq 3} this means that we 
are in the stable range, so we can also omit the codegeneracies, the
contractible path space factors, and thus the coface maps
other than \w[,]{d\sp{0}} which were shown in \wref[.]{eqfone}

Since the map \w{\gamma=2} in \w{\pi\sb{n}\bS{n}}
is invisible in \w[,]{\Hus{-}{\Ft}} it has Adams filtration \w[.]{\geq 1}

However, the diagram
\myzdiag[\label{eqtoda}]{
\bS{n}~\ar[rr]\sp{\gamma=2} \ar@{^{(}->}[d]\sp{\iota} &&
\bS{n}~\ar@{^{(}->}[rr]\sp{\inc} \ar@/^{2.0pc}/[rrrr]\sp{0} &&
\bS{n}\cup\sb{2}\be{n+1} \ar[rr]\sp{\nabla} && \bS{n+1} \\
C\bS{n}\ar[rrrru]\sb{\Id\sb{\be{n+1}}} &&&&&&
}
\noindent shows that one value of the Toda bracket \w{\lra{\nabla,\inc,\gamma}} is
the pinch map \w[,]{\Sigma\bS{n}\cong C\bS{n}/\bS{n}\xra{\cong}\bS{n+1}}
which has degree $1$. Post-composing with \w{\rho\sb{2}\circ q} of 
\wref{eqmooresp} we obtain the value \w{\iota\sb{n+1}\in\pi\sb{n+1}\KP{\Ft}{n+1}}
for the associated secondary cohomology operation 
\w{\lra{\Sq{1},\iota\sb{n},\gamma}} (see \cite[Ch.~1]{TodC}), with indeterminacy
$$
\{0\}~=~2\cdot\pi\sb{n+1}\KP{\Ft}{n+1}+\Sq{1}\sb{\ast}\pi\sb{n}\KP{\Ft}{n}~.
$$
This shows that in fact $\gamma$ has filtration index $1$, as expected.
\end{example}

\end{document}